\documentclass[11pt,a4paper,reqno]{amsart}
\usepackage[utf8]{inputenc}
\usepackage{amsmath,amssymb,amsthm, mathtools}
\usepackage[alphabetic,nobysame]{amsrefs}
\usepackage{enumerate}
\usepackage{xfrac}
\usepackage[symbol,perpage]{footmisc}
\usepackage{hyperref}
\usepackage{enumitem}
\usepackage[british]{babel}
\usepackage{graphicx, subcaption, xcolor}
\usepackage[margin=1in]{geometry}
\usepackage{tikz-cd}
\theoremstyle{plain}
\newtheorem{thm}{Theorem}[section]
\newtheorem{prop}[thm]{Proposition}
\newtheorem{lemma}[thm]{Lemma}
\newtheorem{cor}[thm]{Corollary}

\newtheorem*{claim*}{Claim}
\newtheorem*{thm*}{Theorem}

\theoremstyle{definition}
\newtheorem{defi}[thm]{Definition}
\newtheorem{setting}[thm]{Setting}
\newtheorem*{notation*}{Notation}
\newtheorem*{ack*}{Acknowledgments}

\usepackage{chngcntr}
\counterwithin{figure}{section}

\theoremstyle{remark}

\newtheorem{obs}[thm]{Remark}
\newtheorem*{obs*}{Remark}

\theoremstyle{plain}
\newcounter{main}

 \newtheorem{maintheorem}[main]{Theorem}

\def\Chat{\widehat{\mathbb{C}}}

\def\DD{\mathbb{D}}

\def\dD{\partial\mathbb{D}}

\renewcommand{\tilde}{\widetilde}
\renewcommand{\setminus}{\smallsetminus}

\newcommand\numberthis{\addtocounter{equation}{1}\tag{\theequation}}

\title{Boundaries of multiply connected Fatou components. A unified approach}

\author[G. R. Ferreira]{Gustavo R. Ferreira}
\address{Gustavo R. Ferreira. Centre de Recerca Matem\`atica, Barcelona, Spain}
\email{grodrigues@crm.cat}

\author[A. Jov\'e]{Anna Jov\'e}
\address{Anna Jov\'e. Departament de Matemàtiques i Informàtica, Universitat de Barcelona, Barcelona, Spain}
\email{annajove@ub.edu}
\thanks{Both authors are partially supported by the Spanish State Research Agency through grant number PID2023-147252NB-I00. The first author acknowledges financial support from the Spanish State Research Agency through the Mar\'ia de Maeztu Program for Centers and Units of Excellence in R\&D (CEX2020-001084-M). The second author acknowledges financial support from the Spanish government grant FPI PRE2021-097372.}

\date{\today}

\begin{document}
\begin{abstract}
We analyze the boundaries of multiply connected Fatou components of transcendental maps by means of universal covering maps and associated inner functions. A unified approach is presented, which includes invariant Fatou components (of any type) as well as wandering domains. We prove that any Fatou component admits a harmonic measure on its boundary whose support is the whole boundary. Consequently, we relate, in a successful way, the geometric structure of such Fatou components (in terms of the limit sets of their universal covering maps), the dynamics induced on their boundary from an ergodic point of view, and analytic properties of the associated inner function.
\end{abstract}
\maketitle

\section{Introduction}
Consider the discrete dynamical system generated by a meromorphic map $ f \colon\mathbb{C}\to\widehat{\mathbb{C}} $, i.e. the sequence of iterates $ \left\lbrace f^n(z)\right\rbrace _n $, where $ z\in{\mathbb{C}} $. If the point at infinity is an essential singularity of $ f $, we say that $ f $ is transcendental; otherwise $ f $ extends to $ \widehat{\mathbb{C}} $ as a rational map. More generally, we shall work with functions in the class $ \mathbb{K} $, the smallest class containing all meromorphic functions which is closed under composition (see Sect. \ref{subsect-basic-iteration}).

The Riemann sphere $ \widehat{\mathbb{C}} $, regarded as the phase space of the dynamical system, is divided into two totally invariant sets: the {\em Fatou set} $ \mathcal{F}(f) $,  the set of points $ z\in\mathbb{C} $ such that $ \left\lbrace f^n\right\rbrace _{n\in\mathbb{N}} $ is well-defined and forms a normal family in some neighbourhood of $ z$; and the {\em Julia set} $ \mathcal{J}(f) $, its complement, where the dynamics is chaotic. 
The Fatou set is open and consists in general of infinitely many components, called \textit{Fatou components}. Due to the invariance of the Fatou and Julia sets, Fatou components are either periodic, preperiodic or wandering. Wandering domains do not exist for rational maps \cite{Sullivan}. Periodic Fatou components, are either simply, doubly, or infinitely connected, while wandering domains may exhibit any connectivity \cite{BKL90, Ferreira_MCWD1}.

A standard approach to study the dynamics of $ f $ on the interior of a simply connected Fatou component $ U $ is to associate  $ f|_U $ with  an inner function of the unit disk $ \mathbb{D} $ by means of Riemann maps. More precisely, let $ U $ and $ V $ be simply connected Fatou components of $ f $, $ f(U)\subset V $, and let $ \varphi_U\colon\mathbb{D}\to U $ and $ \varphi_V\colon\mathbb{D}\to V $ be Riemann maps. Then, $ \varphi_V \circ g= f\circ \varphi_U$ for some inner function $ g\colon\mathbb{D}\to\mathbb{D} $, i.e.
\[\begin{tikzcd}
	\DD \arrow{r}{g} \arrow[swap]{d}{\varphi_U}&  \DD  \arrow{d}{\varphi_V} \\	
	U \arrow{r}{f}& V
\end{tikzcd}
\]

 If $ U $ is an invariant Fatou component, we shall take $ U=V $ and $ \varphi_U=\varphi_V $ (denoted just by $ \varphi $), while in the case of wandering domains we consider a sequence of Riemann maps $ \varphi_n\colon\mathbb{D}\to U_n $, with $ f^n(U)\subset U_n $; this gives us the following commutative diagrams. 
\[\begin{tikzcd}
	\DD \arrow{r}{g} \arrow[swap]{d}{\varphi}&  \DD  \arrow{d}{\varphi} &&	\DD \arrow{r}{g_0} \arrow{d}{\varphi_0}&  \DD \arrow{r}{g_1} \arrow{d}{\varphi_1}&  \DD \arrow{r}{g_2} \arrow{d}{\varphi_2}&  \DD \arrow{r}{g_3} \arrow{d}{\varphi_3}&\dots \\	
	U \arrow{r}{f}& U &&U_0 \arrow{r}{f}& U_1\arrow{r}{f}& U_2\arrow{r}{f}& U_3\arrow{r}{f}& \dots
\end{tikzcd}
\] 

In the invariant case, one can apply the classical Denjoy-Wolff theorem to the inner function $ g $ to deduce an alternative proof of the classification of periodic Fatou components (see e.g. \cite[Thm. 6]{Bergweiler93}), and to describe the internal dynamics in Baker domains in a finer way \cite{BF01, FagellaHenriksen06}. In the case of wandering domains, the non-autonomous sequence of inner functions $ \left\lbrace g_n\right\rbrace _n $ allows also the study of the internal dynamics of such Fatou components \cite{befrs1,Ferreira-VanStrien}. 

For simply connected Fatou components, this approach is also useful in studying the dynamics of $ f|_{\partial U} $, by extending the previous conjugacy to $ \partial \mathbb{D} $ in the sense of radial limits and using the theory of prime ends of Carathéodory. In the case of invariant Fatou components, this goes back to the analysis of $ \lambda e^z $  done by  Devaney and Goldberg \cite{DevaneyGoldberg}. Since then, related techniques have been used in a wide range of invariant Fatou components \cite{BakerWeinreich,BakerDominguez, Bargmann,BFJK-Accesses,RipponStallard, BFJK-Escaping, JF23, Jov24-boundaries}. In the case of simply connected wandering domains, similar ideas can be applied, as shown in \cite{BEFRS2}.

In view of the previous discussion, the following question arises naturally. Can one use the same construction with a multiply connected Fatou component $ U $, replacing the Riemann map by a universal covering map $ \pi\colon\mathbb{D}\to U $, in order to study the internal and boundary dynamics of $f|_U$? 

Concerning internal dynamics, this approach was already implemented in \cite{Bonfert} (for general self-maps of multiply connected domains, not necessarily Fatou components), in \cite{BFJK-Connectivity,BFJK_AbsorbingSets} (for Baker domains), and in \cite{Ferreira_MCWD2,FR24} (for wandering domains). Although successful results have been obtained, the increased complexity when working with a universal covering map instead of a Riemann map is apparent.

In this paper, we go one step forward: we are concerned with the boundary extension of the commutative diagram given by  $ \pi_U\colon\DD\to U $, $ \pi_V\colon\DD\to V $, universal covering maps, and  $ g\colon\mathbb{D}\to\mathbb{D} $  such that $ \pi_V \circ g= f\circ \pi_U$, i.e.
\[\begin{tikzcd}
	\DD \arrow{r}{g} \arrow[swap]{d}{\pi_U}&  \DD  \arrow{d}{\pi_V} \\	
	U \arrow{r}{f}& V
\end{tikzcd}
\] 

For a simply connected domain $U$, the radial extension of a Riemann map $\varphi\colon\mathbb{D}\to U$, i.e.  \[\varphi^*\colon\partial\mathbb{D}\to\partial U, \hspace{1cm}\varphi^*(\xi)\coloneqq \lim\limits_{t\to 1^-} \pi(t\xi),\]
is always defined $ \lambda $-almost everywhere, where $\lambda$ denotes the Lebesgue measure on $\partial\mathbb{D}$  (see e.g. \cite[Thm. 17.4]{Milnor}), which enables the study of the boundary dynamics in a measure-theoretical way. This is not always the case for arbitrary multiply connected domains, and depends essentially on the size (in terms of logarithmic capacity) of their complement \cite[Corol. B]{FerreiraJove}. For Fatou components, we prove that the radial extension $ \pi^*\colon\partial\mathbb{D}\to\partial U $ is always well-defined $ \lambda $-almost everywhere. This is the content of our first result.
\begin{maintheorem}{\bf (Radial extension well-defined)}\label{thmA}
	Let $ f\in\mathbb{K} $, and let $ U $ be a Fatou component of $ f $. Then, given a universal covering $ \pi\colon \mathbb{D}\to U $, the radial extension \[\pi^*\colon\partial\mathbb{D}\to\partial U, \hspace{1cm}\pi^*(\xi)\coloneqq \lim\limits_{t\to 1^-} \pi(t\xi),\] is well-defined $ \lambda $-almost everywhere, where $ \lambda $ denotes the normalized Lebesgue measure on $ \partial\mathbb{D} $.
	
\noindent	Equivalently, $ U $ admits a harmonic measure $ \omega_U $. Moreover,  the support of $\omega_U$ is precisely the boundary of $ U $. 
\end{maintheorem}
\begin{obs*}
	Theorem \ref{thmA} also holds for functions of finite type (not necessarily in class $ \mathbb{K} $; see Remark \ref{remark-functions-finite-type}).
Also, it is not entirely new, in the sense that its conclusion was already known for some specific types of Fatou components (see Remark \ref{remark-thmA}). To our knowledge, this is the  first time when all Fatou components have been considered under a unified approach, especially considering the assertion about the support of the harmonic measure.
\end{obs*}

Once the radial extension of the universal covering map is well-defined $ \lambda $-almost everywhere,  it is not difficult to see that \[g\colon\mathbb{D}\to\mathbb{D}, \hspace{1cm} \pi_V\circ g=f\circ \pi_U, \]is an {\em inner function}, that is \[g^*(\xi)\coloneqq \lim\limits_{t\to1^{-}}g(t\xi)\in\partial \mathbb{D}\]
for $ \lambda $-almost every $ \xi\in\partial\mathbb{D} $ (Prop. \ref{prop-inner-function}). We call $ g $ the {\em associated inner function} of $ f|_U $. 

As in the simply connected case, we aim to study the interplay between the analytic properties of the associated inner function, the topology of $ \partial U $, and the dynamics of $ f|_{\partial U} $, in the spirit of \cite{JF23}. Moreover, in the case of multiply connected Fatou components, the geometry of the domain, described by the group of deck transformations and its limit set, plays a significant role. We address the following points.
\begin{itemize}
	\item Concerning the geometry of the domains $ U $ and $ V $, we want to describe the limit sets $ \Lambda_U $ and $ \Lambda_V $ of the corresponding groups of deck transformations. We are also interested in classifying points  $\xi\in \partial\mathbb{D}$ according to the image under $\pi$ of the radius at $ \xi $: whether it is compactly contained in the domain, accumulates on the boundary, or oscillates (as in \cite{FerreiraJove}; see Sect. \ref{subsect-FerreiraJove1}). 
	\item With respect to the associated inner function $ g\colon\mathbb{D}\to\mathbb{D} $, we want to describe their set of singularities $ E(g) $ (i.e. points $g\in\partial\mathbb{D}$ around which $g$ does not extend holomorphically). We also want to characterize for which points on $ \partial \mathbb{D} $  the radial limit $ g^* $ is well-defined. 
	\item Regarding the boundary map $ f|_{\partial U} $, we want to determine its ergodic properties in the spirit of \cite{DM91}. More precisely, for invariant Fatou components, we are interested in exactness (which implies ergodicity) and recurrence. We are also concerned with existence of dense orbits on $ \partial U $.
	
	\noindent In order to deal with the boundary dynamics of a wandering domain, neither ergodicity nor recurrence are appropriate, since it is a non-autonomous system and $ \left\lbrace f^n{(\partial U)}\right\rbrace _n  $ are potentially pairwise disjoint sets. Hence, we are interested in the exactness of $ f|_{\partial U} $. 

\noindent Recall that $ f|_{\partial U} $ is {\em exact} with respect to $\omega_U$  if for every measurable $ X\subset \partial U $, such that for every $ n \geq 1$ we have $ X = f^{-n}(X_n) $ for some measurable $ X_n\subset f^n(\partial U) $, there holds $ \omega_U(X)=0 $ or $ \omega_U(X)=1 $. We say that $ f|_{\partial U} $ has $ n $ {\em exact components} if $ \partial U=\sqcup_{k=1}^n\Sigma_k $ and $ f|_{\Sigma_k} $ is exact for all $ k\geq1 $.
\end{itemize}

First, let us focus on invariant Fatou components of rational maps, which are either basins (of attraction or parabolic) or rotation domains. Basins are either simply or infinitely connected, while rotation domains are either simply connected ({\em Siegel disks}) or doubly connected ({\em Herman rings}). Moreover, the type of inner function is uniquely determined by the type of Fatou component and its connectivity (Prop. \ref{prop-inner-rational}). Concerning multiply connected basins, we prove the following.
\begin{maintheorem}{\bf (Basins of rational maps)}\label{thmB}
	Let $ f $ be a rational map, and let $ U $ be a multiply connected basin. Then,
	\begin{enumerate}[label={\em (\alph*)}]
		\item\label{thmB-itema} the limit set  $ \Lambda$ is  $\partial\mathbb{D} $;
		\item $ g $ has infinite degree, and $ E(g) =\partial\mathbb{D}$;
		\item \label{thmB-itemc} $ (g^*)^n(\xi) $ is well-defined  for all $ n\geq 0 $ if and only if $ \left\lbrace \pi(t\xi)\colon t\in \left[ 0,1\right) \right\rbrace  $ is not compactly contained within $ U $.
	\end{enumerate}
\end{maintheorem}
Moreover, the ergodic properties of the map $ f|_{\partial U} $ were already established by the seminal work of Doering and Mañé \cite[Sect. 6]{DM91}.  In this case, the ergodic properties of the radial extension of the associated inner function (studied in the same paper) transfer directly to $ f|_{\partial U} $, implying that $ f|_{\partial U} $ is exact (in particular, ergodic), and recurrent. In particular, for $ \omega_U $-almost every $ x\in\partial U $, $ \left\lbrace f^n(x)\right\rbrace _n $ is dense on $ \partial U $.

When turning to transcendental meromorphic functions, the structure of the Fatou set is more complicated, due to the presence of an essential singularity. Indeed, new types of Fatou components appear, namely {\em Baker domains} (i.e. invariant Fatou components in which iterates converge to an essential singularity) and wandering domains. Moreover, doubly connected Fatou components are not just Herman rings, but also wandering domains or even basins of attraction. This leads to a less transparent relationship between the group of deck transformations, the associated inner function, and the boundary map, as the following examples illustrate, which are detailed in Section \ref{section-pathological-boundaries} and for which we can compute explicitly the associated inner functions.

If we let first  $ U $ be a Herman ring, then the associated inner function $ g $ is a hyperbolic Möbius transformation (and $ g|_{\partial\mathbb{D}} $ is non-ergodic and non-recurrent). However, $ f $ is ergodic and recurrent when restricted to any of the two connected components of $\partial U$. Hence, it follows that the ergodic properties of the associated inner function do not always correspond to the ones of $ f|_{\partial U} $. 

 As a second example, let $ U $ be a doubly connected wandering domain of an entire function. Then, the associated inner function can be taken to be the identity, but $f$ is exact when restricted to any of the two connected components of $\partial U$. Together with the example of the Herman ring, we observe that the same type of associated inner function may lead to different ergodic properties of the boundary map $ f|_{\partial U} $.
 
 As a final example, consider the map 
 $ f(z)=e^{\alpha(z-1/z)} $, with $ \alpha\in (0, 1/2) $, which is a self-map of $\mathbb{C}^*$, studied by  Baker \cite[Thm. 2]{Baker-puncturedplane}. The Fatou set $ \mathcal{F}(f) $ consists of the immediate attracting basin $U$ of 1, which is doubly connected, and $f|_U$ has infinite degree.  We provide in Section \ref{subsect-baker} a detailed study of such basin of attraction, which shows that Theorem \ref{thmB} can fail completely in the transcendental case.

Despite the difficulties outlined above, we can still prove the following. As usual when dealing with Fatou components of transcendental functions, some control on the singular values is needed (see Sect. \ref{subsection-SV}); we denote by $SV(f|_U)$ the set of singular values of $f|_U$.

\begin{maintheorem}{\bf (Basins in class $ \mathbb{K} $)}\label{thmC}
		Let $ f \in\mathbb{K}$, and let $ U $ be a multiply connected basin.  Then,
				\begin{enumerate}[label={\normalfont(\alph*)}]
			\item the associated inner function $ g $ has infinite degree, and $ \Lambda\subset E(g) $. Moreover, if $ U $ is infinitely connected, then $ E(g) $ is uncountable.
			\item $ f|_{\partial U} $ is exact (and, in particular, ergodic) and recurrent. Moreover, for $ \omega_U $-almost every $ x\in\partial U $, $ \left\lbrace f^n(x)\right\rbrace _n $ is dense on $ \partial U $.
		\end{enumerate}
		Moreover, if $ SV(f|_U) $ is compactly contained in $U$, then
		\begin{enumerate}[label={\normalfont(\alph*)}]
			\setcounter{enumi}{2}
			\item either $ \lambda (E(g))=  \lambda (\Lambda)=0 $, or $ E(g)=\Lambda=\partial\mathbb{D} $.
		\end{enumerate}
\end{maintheorem} 

\begin{obs*} Although Baker domains are not considered in this paper, it is easy to see that Theorem \ref{thmC} holds for a particular type of Baker domains, the so-called {\em doubly parabolic} ones (see the discussion in Sect. \ref{section-basins-trans}, Rmk. \ref{remark-dp-Baker-domains}).
\end{obs*}

Finally, we consider wandering domains of functions in class $ \mathbb{K} $. Recent results by Martí-Pete, Rempe and Waterman \cite{MartipeteRempeWaterman,MartiPeteRempeWatermann-mero}, and Evdoridou, Martí-Pete and Rempe \cite{EvdoridouMartipeteRempe}  show that wandering domains of meromorphic functions exhibit a quite large flexibility, and any non-autonomous sequence of holomorphic maps between pairwise disjoint compact sets can be turned into a sequence of wandering domains of a meromorphic function,  being such function constructed using approximation theory. With this in mind, we prove the following.
\begin{maintheorem}{\bf (Wandering domains in class $ \mathbb{K} $)}\label{thmD}
		Let $ f \in\mathbb{K}$, and let $ \left\lbrace U_n\right\rbrace _n $ be a sequence of wandering domains, such that $ U_0 $ is multiply connected and $ f^n(U_0)\subset U_n $. Let $ \pi_n\colon\mathbb{D}\to U_n $ be universal covering maps, so that  the inner functions $ \left\lbrace g_n\right\rbrace _n $  such that $ \pi_{n+1}\circ g_n=f\circ \pi_n $  satisfy $ g_n(0)=0 $, for all $ n\geq 0 $. Then,
	\begin{enumerate}[label={\em (\alph*)}]
		\item\label{thmD-a} if $U_0$ is infinitely connected, then its limit set $ \Lambda_0 $ is either a Cantor set of zero measure, a Cantor set of positive measure, or $ \partial\mathbb{D} $. 
		
		\noindent For each of these possibilities, there exists a transcendental meromorphic function with a multiply connected wandering domain  with a limit set of the corresponding type.
		\end{enumerate}
	If, furthermore, $U_n$ is bounded for all $ n\geq 0 $, then
		\begin{enumerate}[label={\em (\alph*)}]   \setcounter{enumi}{1}
		\item the limit set $ \Lambda_n$ of $ U_n $ satisfies that $ \lambda (\Lambda_n)= \lambda (\Lambda_0)$, for all $ n\geq 0 $.
		\item If $ \sum_n 1-\left| g'_n(0)\right| =\infty $, then $f|_{\partial U_n}$ is exact and $ \Lambda_n=\partial\mathbb{D} $, for all $ n\geq 0 $.
	\end{enumerate}
\noindent Moreover, there exists a transcendental meromorphic function with a wandering domain whose exact components are singletons.
\end{maintheorem} 

In Section \ref{section-WD}, we will see that more can be said about multiply connected wandering domains of entire functions, for which the geometry and dynamics is known to be much more restrictive \cite{BakerWD3,KS08,BergweilerRipponStallard_MCWD,RipponStallard-MCWDMero,RipponStallard-Eremenko,Ferreira_MCWD1,FR24}.

\begin{ack*}
We would like to thank Núria Fagella, Lasse Rempe, Phil Rippon, and Gwyneth Stallard for interesting discussions and comments. We would also like  to thank Linda Keen for helping us to find a copy of the paper \cite{Keen}.
\end{ack*} 
\section{Iteration in class $ \mathbb{K} $ and associated inner functions}

Consider $ f\in\mathbb{K} $, i.e. \[f\colon \widehat{\mathbb{C}}\smallsetminus E(f)\to \widehat{\mathbb{C}},\] where $ \Omega(f)\coloneqq\widehat{\mathbb{C}}\smallsetminus E(f)  $ is the largest set where $ f $ is meromorphic and $ E(f) $ is the set of singularities of $ f $, which is assumed to be closed and countable. Note that $ \Omega(f)$ is open. 

\begin{notation*}
	Having fixed a function $ f\in\mathbb{K} $, we denote $ \Omega(f) $ and $ E(f) $ simply by $ \Omega $ and $ E $, respectively. Given a domain $ U\subset\Omega $, we denote by $ \partial U $ the boundary of $ U $ in $ \Omega $, and we keep the notation $ \widehat{\partial} U $ for the boundary with respect to $ \widehat{\mathbb{C}} $.
\end{notation*}

Maps in class $ \mathbb{K} $ satisfy the following analytic property.
\begin{lemma}{\bf (Ahlfors islands property, {\normalfont\cite[Example 1]{BakerDominguezHerring}})}\label{lemma-ahlfors}
	Let $ f \in\mathbb{K}$. Then, $ f $ satisfies the {\em $ 5 $-island property}. That is,  for every $ z_0 \in E(f)$,  given any neighbourhood $ U $ of $ z_0 $ and 5 simply connected disjoint Jordan domains $ D_1 , \dots, D_5$ in $ \widehat{\mathbb{C}} $, there exists $ D\subset U $ such that $ f $ maps conformally $ D $ onto $ D_i $ for some $ i\in \left\lbrace 1,\dots, 5\right\rbrace  $.
\end{lemma}

\subsection{Basic results on iteration}\label{subsect-basic-iteration}
The dynamics of functions in class $ \mathbb{K} $ has been studied in \cite{Bolsch-Fatoucomponents, BakerDominguezHerring}. The standard theory of Fatou and Julia for rational or entire functions extends successfully to this more general setting. The Fatou set $ \mathcal{F}(f) $ is defined as the largest open set in which $ \left\lbrace f^n\right\rbrace _n $ is well-defined and normal, and the Julia set $ \mathcal{J}(f) $, as its complement in $ \widehat{\mathbb{C}} $. We need the following properties.

\begin{thm}{\bf (Properties of Fatou and Julia sets, {\normalfont \cite[Thm. A]{BakerDominguezHerring}})}\label{thm-dynamicsK}
	Let $ f\in\mathbb{K} $. Then,
	\begin{enumerate}[label={\em (\alph*)}]
		\item\label{Fatou1} $ \mathcal{F}(f) $ is completely invariant in the sense that $ z\in\mathcal{F}(f) $ if and only if $ f(z)\in \mathcal{F}(f) $;
		\item\label{Fatou2} for every positive integer $ n $, $ f^n\in\mathbb{K} $, $ \mathcal{F}(f^n)=\mathcal{F}(f) $ and $ \mathcal{J}(f^n)=\mathcal{J}(f) $;
	\item\label{Fatou3}  $ \mathcal{J}(f) $ is perfect.
	\end{enumerate}
\end{thm}
Note that, if $ E(f)=\emptyset $, then $ f $ is a rational map.  If $ E(f)\neq\emptyset $, we always assume $ \infty\in E(f) $. Except in the case when $ f $ is a transcendental entire function or conjugate to a self-map of $ \mathbb{C}^* $, points which are eventually mapped to a singularity are dense in $ \mathcal{J}(f) $, as shows the following lemma.
\begin{lemma}{\bf (Characterization of the Julia set, {\normalfont \cite[p. 650]{BakerDominguezHerring}})}\label{lemma-julia-set}
	Let $ f\in\mathbb{K} $ be such that $ \infty $ is an essential singularity and which has at least one non-omitted pole.
	Then, for all $ n\geq 0 $,
	\[ E(f^n)= E(f)\cup f^{-1}(E(f))\cup\dots \cup f^{-(n-1)}(E(f)) \]
	and \[\mathcal{J}(f)=\overline{\bigcup\limits_{n\geq 0} E(f^n)}.\]
\end{lemma}

By Theorem \ref{thm-dynamicsK}{\em\ref{Fatou1}}, Fatou components (i.e. connected components of $ \mathcal{F}(f) $) are mapped among themselves, and hence classified into periodic, preperiodic, or wandering. By Theorem \ref{thm-dynamicsK}{\em\ref{Fatou2}}, the study of periodic Fatou components reduces to the invariant ones, i.e. those for which $ f(U)\subset U $. Those Fatou components are classified into  basins of attraction, parabolic basins, Siegel disks, Herman rings, and Baker domains \cite[Thm. C]{BakerDominguezHerring}. A \textit{Baker domain} is, by definition, a periodic Fatou component $ U $ of period $ k\geq 1 $ for which there exists $ z_0\in\widehat{\partial} U$ such that $ f^{nk}(z)\to z_0 $, for all $ z\in U $ as $ n\to\infty $, but $ f^k $ is not meromorphic at $ z_0 $. In this case, $ z_0 $ is accessible from $ U $  \cite[p. 658]{BakerDominguezHerring}.

\begin{thm}{\bf (Connectivity of Fatou components, {\normalfont \cite{Bolsch-Fatoucomponents}})} Let $ f\in\mathbb{K} $, and let $ U $ be a periodic Fatou component of $ f $. Then, the connectivity of $ U $ is 1, 2, or $ \infty $.\end{thm}

Wandering domains may exhibit any connectivity  \cite{BKL90,Ferreira_MCWD1}.

\subsection{Harmonic measure with respect to a Fatou component. Theorem \ref{thmA}}

Let $ U $ be a (multiply connected) Fatou component for $ f\in \mathbb{K} $, and let $ \pi\colon \mathbb{D}\to U $ be a universal covering map.
Our first goal  is to prove Theorem \ref{thmA}, which asserts  that the {\em radial extension} $ \pi^*\colon\partial \mathbb{D}\to\partial U  $, where $$ \pi^*(\xi)\coloneqq\lim \limits_{t\to 1^-}\pi (t\xi) $$ is the {\em radial limit} of $ \pi_U $ at $ \xi\in\partial \DD $, is well-defined $ \lambda $-almost everywhere, where $ \lambda $ denotes the Lebesgue measure on $ \partial\mathbb{D} $. 

Recall that, for a hyperbolic domain $ U $, the existence of a radial extension for $ \pi $ is equivalent to $ \Chat\smallsetminus U $ having positive logarithmic capacity (\cite[Chap. VII.5]{Nev70}, see also \cite[Corol. B]{FerreiraJove}). 
In this case, setting $ \pi(0)=p\in U $, one defines the {\em harmonic measure} with respect to the domain $ U $ and with basepoint $ p\in U $, $ \omega_U(p, \cdot) $, as the push-forward measure of $ \lambda $ under the radial extension $ \pi^* $. More precisely, given a Borel set $ A\subset\widehat{\mathbb{C}}$, \[\omega_U(p, A)\coloneqq\lambda ((\pi^*)^{-1}(A)).\]
Note that the harmonic measure $ \omega_U(p,\cdot) $ is well-defined. Indeed, the set \[(\pi^*)^{-1}(A)=\left\lbrace \xi\in\partial\mathbb{D}\colon\pi^*(\xi)\in A\right\rbrace \]is a Borel set of $\partial \mathbb{D} $ \cite[Prop. 6.5]{Pom92}, and hence measurable. 
We also note that the definition of $ \omega_U(p, \cdot) $ is independent of the choice of $ \pi$, provided it satisfies $ \pi(0)=p$, since $ \pi$ is uniquely determined up to pre-composition with a rotation of $ \mathbb{D} $, and $ \lambda $ is invariant under rotation. 

We say that $ U $ admits a {\em harmonic measure} $ \omega_U $ if the previous construction can be carried out.  Theorem \ref{thmA} states that Fatou components always admit a harmonic measure. 
It also asserts that the support of the harmonic measure is precisely $ \widehat{\partial} U $, which leads to interesting ergodic properties, as we will see.

\begin{obs}\label{remark-thmA}
	Let us note that the content of Theorem \ref{thmA} is already well-known for certain classes of functions (although the general statement is new). 
	First, if $ f\colon\widehat{\mathbb{C}}\to\widehat{\mathbb{C}} $ is a rational map, then $ \widehat{\partial} U $ is a uniformly perfect set \cite{ManeRocha, Hinkkanen_rat, Hinkkanen_polynomials}, and this already implies that $ \widehat{\partial} U $ has positive logarithmic capacity and the support of the corresponding harmonic measure is $ \widehat{\partial} U $ \cite[Thm. 1]{Pommerenke_UnifPerfect}. Hence, Theorem \ref{thmA} holds when $ f $ is rational (see also \cite[Sect. 6]{DM91}). 
	
	Second, when $ f\colon{\mathbb{C}}\to{\mathbb{C}} $ is a transcendental entire function, all periodic Fatou components are simply connected \cite{Baker84}, and hence satisfy Theorem \ref{thmA}. Multiply connected wandering domains of entire functions are known to admit a harmonic measure  \cite[Thm. 1.6(c)]{BergweilerRipponStallard_MCWD}. Meromorphic functions with one omitted pole can be conjugate to self-maps of $ \mathbb{C}^* $, for which all Fatou components are simply or doubly connected \cite[Thm. 1]{Bak87}, and thus Theorem \ref{thmA} holds trivially.
\end{obs}

\begin{proof}[Proof of Theorem \ref{thmA}]
	According to Remark \ref{remark-thmA}, it is enough to prove the theorem for $ f\in\mathbb{K} $ for which $ \infty $ is an essential singularity and having at least one pole which is not omitted.
	
	First we have to see that $ \widehat{\mathbb{C}}\smallsetminus U $ has positive logarithmic capacity. Since sets of zero logarithmic capacity have zero Hausdorff dimension \cite[Thm. 10.1.3]{Pom92}, it suffices to  prove that $ \mathcal{J}(f) $ has positive Hausdorff dimension,  for $ f\in \mathbb{K} $. This was proved for meromorphic functions by Stallard \cite{Stallard}. As it is explained in \cite{Stallard}, the proof relies on constructing an iterated function system generated by inverse branches of $ f $ whose limit set is contained in the Julia set, by means of the 5-island property. By Lemma \ref{lemma-ahlfors}, the 5-island property holds for any  $ f\in\mathbb{K} $, so the whole construction can be carried out. We leave further details to the reader.
	
	We are left to prove that $ \textrm{supp } \omega_U=\widehat{\partial} U $, that is, for $ x\in\widehat\partial U $ and $ r>0 $, we shall see that $ \omega_U(D(x,r))>0 $. We shall distinguish two cases, if $ U\subsetneq\mathcal{F}(f) $ or if $ U=\mathcal{F}(f) $. 
	
	First, assume $ U\subsetneq\mathcal{F}(f) $. 
	We will apply the following result to appropriate domains.
	\begin{lemma}{\bf (Comparison principle for harmonic measure, {\normalsize \cite[Corol. 21.1.14]{Conway2}})}
		Let $U, \ \widetilde{U}\subset\widehat{\mathbb{C}} $ be  hyperbolic domains, with $ U\subset \widetilde{U}$. Assume $ \widetilde{U} $ admits a harmonic measure. Then, for all $ z\in U $ and any measurable set $ B\subset {\partial}\widetilde{U}\cap\partial U $, 
		\[\omega_U(z, B) \leq \omega_{\widetilde{U}} (z, B).\]
	\end{lemma}
	Note that, by taking $B = \partial U\cap\partial\widetilde{U}$, we obtain
	\begin{equation}\label{eq:harmonic}
	\omega_{\widetilde{U}}(z,{\partial}\widetilde{U}\smallsetminus \partial U) \leq \omega_U(z,{\partial}U\smallsetminus {\partial}\widetilde{U}).
	\end{equation}
	
	Let us choose the domain $ \widetilde{U} $ as follows. By the blow-up property of the Julia set, one can find a topological disk $ V\subset D(x,r) $ such that $\overline{V}$ belongs to a Fatou component different from $ U $. Then, let \[\widetilde{U}\coloneqq U\cup (D(x,r)\smallsetminus \overline V).\]
	
	Clearly, \begin{align*}
		&	 U\subset  \widetilde{U},\\
		&	 \partial U\smallsetminus \partial  \widetilde{U} = \partial U\cap D(x,r), \\
		&	 \partial  \widetilde{U}\smallsetminus \partial U  \supset \partial V. \numberthis \label{eq:inclusion}\\
	\end{align*}
	By (\refeq{eq:harmonic}) and (\refeq{eq:inclusion}),
	\[\omega_{\widetilde{U}}(z, \partial V)\leq  \omega_{\widetilde{U}}(z,\partial\widetilde{U}\smallsetminus \partial U ) \leq \omega_U(z,{\partial}U\smallsetminus {\partial}\widetilde{U})= \omega_U(z,\partial U\cap D(x,r)) .\] Moreover, $0<\omega_{\widetilde{U}}(z, \partial V)$ since $ \partial V $ is a non-degenerate continuum. Thus, $$ 0<\omega_U(z,\partial U\cap D(x,r)) $$  as desired.
	
	Now, assume $ U=\mathcal{F}(f) $.
	As before, let us restrict to the case when $ \infty $ is an essential singularity with one non-omitted pole. It is clear that the iterated function system considered by Stallard \cite{Stallard} leads to the existence of five disjoint Jordan domains $ D_1,\dots, D_5 $ such that $$D_i\cap \mathcal{J}(f)=D_i\cap \partial U, \hspace{0.5cm}i=1,\dots, 5,$$ has positive Hausdorff dimension (and hence, positive logarithmic capacity). Then, take $ n\geq1 $ such that $ E(f^n)\cap D(x,r)\neq \emptyset $ (which exists by Lemma \ref{lemma-julia-set}). Note that $ f^n\in\mathbb{K} $, and hence satisfies the Ahlfors islands property (Lemma \ref{lemma-ahlfors}). Then, there exists $ D\subset D(x,r) $ and $ F_n $ inverse branch of $ f^n $ such that $ F_n(D_i)=D $, for some $ i\in \left\lbrace 1,\dots, 5\right\rbrace  $. Then,
	$ D \cap\mathcal{J}(f)=D\cap\partial U$ has positive Hausdorff dimension, and hence positive logarithmic capacity. Therefore, $ \omega_U(D(x,r))>0 $, as desired. This ends the proof of the theorem.
\end{proof} 
\begin{cor}\label{cor-radial-extension} Let $ f\in\mathbb{K} $, let $ U $ be a Fatou component of $ f $, and let $\pi\colon\mathbb{D}\to U$ be a universal covering map. Let
	\[\Theta_\Omega\coloneqq \left\lbrace \xi\in\partial\mathbb{D}\colon \pi^*(\xi) \ \textrm{\em   exists and belongs to }\Omega(f) \right\rbrace. \] Then, $ \lambda (\Theta_\Omega)=1 $.
\end{cor}
\begin{proof}
	According to Theorem \ref{thmA}, for $ \lambda $-almost every $ \xi\in\partial\mathbb{D} $, the radial limit exists and belongs to $ \widehat{\partial }U $. Hence, it is enough to prove that the set
	\[\Theta_E\coloneqq \left\lbrace \xi\in\partial\mathbb{D}\colon \pi^*(\xi)\in E(f)\right\rbrace \] has zero $ \lambda $-measure. But this follows immediately from the fact that $ E(f) $ is countable, and $ \pi^* $ achieves different values $ \lambda $-almost everywhere (by Privalov's Theorem, see e.g. \cite[Sect. 6.1]{Pom92}). Hence, $ \Theta_E $ is a countable union of sets of zero measure, so it has zero measure.
\end{proof}

\begin{obs}{\em (Functions of finite type)}\label{remark-functions-finite-type}
	The theory of Fatou and Julia extends successfully to maps of finite type (i.e. having only finitely many singular values), as developed by Epstein \cite{Epstein-thesis} (see also \cite{Astorg-thesis}). All Fatou components of such maps are eventually periodic \cite[Thm. 7, p. 148]{Epstein-thesis}, so it makes sense to ask whether Theorem \ref{thmA} also holds in this case. The answer is yes, since finite type maps satisfy the Ahlfors Island property \cite[Prop. 9, p. 88]{Epstein-thesis}, and hence the previous proof can be carried out without any obstacle. Moreover, it is easy to see that Theorem \ref{thmC} (stated for maps in class $\mathbb{K}$, to be proved in Sect. \ref{section-basins-trans}) holds also in this setting, since the proof only depends on the map restricted to the basin, not on the global features of the function.
\end{obs}

\subsection{The associated inner function and its radial extension}

We aim to study multiply connected Fatou components, either periodic (which we assume to be invariant) or wandering. More precisely, if $ U $ is an invariant Fatou component and $ \pi\colon\mathbb{D}\to U $ is a universal covering map, one gets the following commutative diagram.
\[\begin{tikzcd}
	\DD \arrow{r}{g} \arrow[swap]{d}{\pi}&  \DD  \arrow{d}{\pi} \\	
	U \arrow{r}{f}& U
\end{tikzcd}
\] 
 In the case when $ U $ is a wandering domain, let $ U_n $ be the Fatou component which contains $ f^n(U) $, and consider universal covering maps $ \pi_n\colon\DD\to U_n $.
\[\begin{tikzcd}
	\DD \arrow{r}{g_0} \arrow{d}{\pi_0}&  \DD \arrow{r}{g_1} \arrow{d}{\pi_1}&  \DD \arrow{r}{g_2} \arrow{d}{\pi_2}&  \DD \arrow{r}{g_3} \arrow{d}{\pi_3}&\dots \\	
	U_0 \arrow{r}{f}& U_1\arrow{r}{f}& U_2\arrow{r}{f}& U_3\arrow{r}{f}& \dots
\end{tikzcd}
\] 
In order to fit both situations (invariant and wandering), we shall work under the following setting.

\begin{setting}\label{setting} Let $ f\in\mathbb{K} $, and  let $ U$, $V $ Fatou components with $ f\colon U\to V $.
	Consider  $ \pi_U\colon\DD\to U $, $ \pi_V\colon\DD\to V $ universal covering maps, and let $ g\colon\mathbb{D}\to\mathbb{D} $ be such that $ \pi_V \circ g= f\circ \pi_U$.
	\[\begin{tikzcd}
		\DD \arrow{r}{g} \arrow[swap]{d}{\pi_U}&  \DD  \arrow{d}{\pi_V} \\	
		U \arrow{r}{f}& V
	\end{tikzcd}
	\] 
\end{setting}

{The following proposition, which is the cornerstone of all the subsequent construction,  has appeared in special cases before;  see e.g. \cite[Prop. 1.1]{EvdoridouVasiliki2020FA}, \cite[Prop. 5.6]{Jov24}, and \cite[Lemma 3.2]{Ferreira_MCWD2}.}

\begin{prop}\label{prop-inner-function}
	In the Setting \ref{setting},	$ g $ is an {\em inner function}, i.e. 	a holomorphic self-map of the unit disk $ g\colon\mathbb{D}\to\mathbb{D} $ such that, for $ \lambda $-almost every point  $\xi\in \partial\mathbb{D} $, $$ g^*(\xi)\coloneqq\lim\limits_{t\to 1^-}g(t\xi)  \in\partial \mathbb{D} .$$
\end{prop}
We say that $ g $ is an {\em inner function associated with} $ f|_U $. Note that any two inner functions associated with $ f|_U $ are conformally equivalent, since universal covering maps are uniquely determined up to precomposition by an automorphism of $\mathbb{D}$. Before proving Proposition \ref{prop-inner-function}, we need the following lemma.
\begin{lemma}{\bf (Radial limits commute)}\label{lemma-radial-limits}
	In the Setting \ref{setting}, consider $\xi\in\Theta_\Omega$, i.e. such that 
	$\pi_U^*(\xi)$ exists and belongs to $ \Omega(f) $.  Then, $g^*(\xi)$ and $\pi_V^*(g^*(\xi))$ exist, and
	\[ f(\pi_U^*(\xi)) = \pi_V^*(g^*(\xi))\in\widehat{\partial} V. \]
\end{lemma}
\begin{proof} 
	
	Let $  R_\xi(t)\coloneqq t\xi$, $  t\in \left[ 0,1\right)  $. By assumption, $$ \pi_U( R_\xi (t))\to \pi^*_U(\xi)\eqqcolon w\in\partial U\subset \Omega(f),$$  as $ t\to 1^- $. Since $ f $ is continuous at $ w $ and $ \pi_V\circ g=f\circ \pi_U $, \[\pi_V(g( R_\xi (t)))=f(\pi_U (R_\xi (t)))\to f(w)\in\widehat\partial V, \]as $ t\to 1 ^-$. Thus, $ f(w) $ is an accessible point in $ \widehat{\partial} V $. Hence, there exists a point $\zeta\in\partial\DD$ with $\pi_V^*(\zeta) = f(w)$ (see \cite[Prop. 5.3]{FerreiraJove}), and furthermore  the curve $g(R_\xi)$ lands at $\zeta$. By the Lehto-Virtanen Theorem (see e.g. \cite[Sect. 4.1]{Pom92}) applied to the curve $g(R_\xi)$, we have that $g^*(\xi) = \zeta$, as desired.
\end{proof}

\begin{proof}[Proof of Proposition \ref{prop-inner-function}]
	Since $ g $ is a holomorphic self-map of the unit disk, its radial extension is well-defined $ \lambda $-almost everywhere (see e.g. \cite[Thm. A.3]{Milnor}). We need to show that $ \left| g^*(\xi)\right| =1 $ for $ \lambda $-almost every $ \xi\in\partial\mathbb{D} $. 
	
	Assume, on the contrary, that there exists $ A\subset\partial\mathbb{D} $ with $ \lambda (A)>0 $ and $ \left| g^*(\xi)\right| <1 $ for all $ \xi\in A $. By Corollary \ref{cor-radial-extension}, we can assume without loss of generality, that $ \pi_U^*(\xi)\subset\partial U\subset \Omega(f) $ for all $ \xi\in A $. Then, \[ f(\pi_U^*(\xi))=\pi_V(g^*(\xi))\subset V\subset \mathcal{F}(f), \] but $ \pi_U^*(\xi)\in\partial U\subset \mathcal{J}(f)$. This is a contradiction due to the total invariance of the Fatou and Julia sets.
\end{proof}

 Throughout this paper we denote the set of {\em singularities} of $ g $ (i.e. $ \xi\in \partial \mathbb{D} $ such that $ g $ does not extend holomorphically to a neighbourhood of $ \xi $) by $ E(g) $.  Note that the radial limit $ g^*(\xi) $ may exist even in the case when $ \xi\in E(g) $ (in particular, we may have $ E(g)=\partial\mathbb{D} $, but the radial extension $ g^*\colon\partial\mathbb{D}\to \partial\mathbb{D}$ is always well-defined $ \lambda $-almost everywhere). We also assume that any inner function $ g $ is always continued to $ \widehat{\mathbb{C}}\smallsetminus \overline{\mathbb{D} }$ by the reflection principle,
and to $ \partial\mathbb{D}\smallsetminus E(g) $ by analytic continuation. In other words, $ g $ is considered as a meromorphic function $$ g\colon   \widehat{\mathbb{C}}\smallsetminus E(g)   \longrightarrow \widehat{\mathbb{C}}.$$

\subsection*{Iteration of inner functions. Ergodic properties of $ g^*\colon\partial\mathbb{D}\to\partial\mathbb{D} $}
One can consider the measure-theoretical dynamical system induced by the radial extension \[ g^*\colon\partial\mathbb{D}\to\partial\mathbb{D},\] which is well-defined $ \lambda $-almost everywhere. The following is a recollection of ergodic properties of $ g^* $, which we will need in the following sections. As we will see, this ergodic classification depends essentially on the type of inner function given by Cowen's classification \cite{cowen}. Recall that an inner function is {\em elliptic} if and only if its Denjoy-Wolff point lies in $\mathbb{D}$. Inner functions with Denjoy-Wolff point in $\partial\mathbb{D}$ are classified into {\em doubly parabolic}, {\em hyperbolic} and {\em simply parabolic}, according to its internal dynamics \cite{cowen}.

\begin{thm}{\bf (Ergodic properties of $ g^* $)}\label{thm-ergodic-properties}
	Let $ g\colon \mathbb{D}\to\mathbb{D} $ be an inner function, with Denjoy-Wolff point $ p\in\overline{\mathbb{D}} $, and let $ g^*\colon \partial\mathbb{D}\to\partial\mathbb{D} $ be its radial extension to the unit circle. Then, the following hold. \begin{enumerate}[label={\em (\alph*)}]
		
		\item  \label{ergodic-properties-nonsing} $ g^* $ is  non-singular. In particular, for $ \lambda $-almost every $ \xi\in\partial\mathbb{D} $, its infinite orbit under $ g^* $, $ \left\lbrace (g^n)^*(\xi)\right\rbrace _n $, is well-defined.
		
		\item   \label{ergodic-properties-a} $ g^* $ is exact if and only if $ g $ is elliptic or doubly parabolic.
		
		\item  \label{ergodic-properties-b} If $ g^* $ is recurrent, then it is ergodic. In this case, for every $ A\in \mathcal{B} (\mathbb{D})$ with $ \lambda(A)>0 $, we have that for $ \lambda $-almost every $ \xi\in\partial\mathbb{D} $, there exists  a sequence $ n_k\to\infty $ such that $  (g^{n_k})^*(\xi)\in A $. In particular, for $ \lambda $-almost every $ \xi\in\partial\mathbb{D} $, $ \left\lbrace (g^n)^*(\xi) \right\rbrace_n  $ is dense in $ \partial\mathbb{D} $.
		\item \label{ergodic-properties-c} If $ g $ is an elliptic inner function such that $ g(0)=0 $, the Lebesgue measure $ \lambda  $ is invariant under $ g^* $. In particular, $ g^* $ is recurrent.
	\end{enumerate}
\end{thm}
The previous theorem compiles different results in \cite{Aaronson97, DM91, BFJK-Escaping}. A detailed explanation can be found in \cite[Thm. 3.10]{Jov24}. 

\begin{lemma}{\bf (Elliptic and doubly parabolic inner functions)}\label{lemma-elliptic-inner}
		Let $ g\colon \mathbb{D}\to\mathbb{D} $ be an inner function, elliptic or doubly parabolic. Let $I\subset\partial\mathbb{D}$ be an open circular interval. Then, there exists $n\geq 0$ such that $ (g^*)^n(I) = \partial\mathbb{D} $.
\end{lemma}

\begin{proof} If $g$ has finite degree, then it is a rational map and $\mathcal{J}(f)=\partial\mathbb{D}$. Then, the statement follows from the blow-up property of the Julia set \cite[Corol. 14.2]{Milnor}.
	
In the infinite degree case, the proof follows from	combining \cite[Lemma 8]{BakerDominguez},  \cite[Thm. 2.24]{Bargmann} and \cite[Thm. A]{BFJK_AbsorbingSets}. Indeed, we have that, for elliptic or doubly parabolic inner functions, \[\partial \mathbb{D}=\overline{\bigcup\limits_{n\geq 0} E(g^n)}.\] In particular, there exists $ \xi\in E(g^n) \cap I$. Then, $ (g^*)^n(I) = \partial\mathbb{D} $, as desired.
\end{proof}
\section{The main construction: relating inner functions, limit sets and universal covering maps} \label{section-main-construction}

This section is devoted to the main results of the paper, although they are somewhat technical. We aim to extend the commutative diagram in Setting \ref{setting} to the boundary. That is, let $ f\in\mathbb{K} $, and $ U$, $V $ Fatou components with $ f\colon U\to V $.
Consider  $ \pi_U\colon\DD\to U $ and $ \pi_V\colon\DD\to V $ universal covering maps, and let $ g\colon\mathbb{D}\to\mathbb{D} $ be such that $ \pi_V \circ g= f\circ \pi_U$, i.e.
\[\begin{tikzcd}
	\DD \arrow{r}{g} \arrow[swap]{d}{\pi_U}&  \DD  \arrow{d}{\pi_V} \\	
	U \arrow{r}{f}& V
\end{tikzcd}
\] 
Although we already proved a measure-theoretical extension in terms of radial limits (Lemma \ref{lemma-radial-limits}), we want a more in-depth analysis which leads to a meticulous understanding of the mapping properties. To do so, we will study the interplay between the analytic properties of the associated inner function, the topology of $ \partial U $, and the dynamics of $ f|_{\partial U} $, as done in the simply connected case in \cite{JF23}. Moreover, we take into account the geometry of the domain, described by the group of deck transformations and its limit set, using the tools developed in \cite{FerreiraJove}.

The plan is as follows: first, we gather all the concepts introduced in \cite{FerreiraJove} to describe the boundary behaviour of a universal covering map which are relevant to the current paper; and second, we study the interplay between such boundary behaviour and the previous commutative diagram.

\subsection{Background results. Boundary behaviour of universal covering maps}\label{subsect-FerreiraJove1}

Given a multiply connected domain $ U$, let $ \pi_U\colon\mathbb{D}\to U$ be a universal covering map and  denote by $ \Lambda_U $ the limit set. In the sequel, when we are dealing only with one domain (like in this section), we shall write just $\pi$ and  $ \Lambda $ to denote the universal covering map and the limit set, respectively. 
	Let $ \xi\in\partial \mathbb{D} $; the following definitions are standard.
\begin{itemize}
	
	\item If $ \eta\colon \left[ 0,1\right) \to\mathbb{D} $ is a curve  landing at $ \xi $, the {\em cluster set} $ Cl_\eta(\pi, \xi) $ of $ \pi $ at $ \xi $ along $ \eta $ is defined as the set of values $ w\in\widehat{\mathbb{C}}$ for which there is an increasing sequence $ \left\lbrace t_n\right\rbrace _n \subset (0,1) $ such that $ t_n\to 1 $ and $\pi(\eta(t_n))\to w $, as $ n\to\infty $.

	\noindent In the particular case when $ \eta=R_\xi= \left\lbrace t\xi \colon t\in \left[ 0,1\right)  \right\rbrace  $, we say that $ Cl_\eta(\pi, \xi) $ is  the {\em radial cluster set} of $ \pi $ at $ \xi $, and denote it by $ Cl_R(\pi, \xi) $.
	\item The {\em angular cluster set} $ Cl_\mathcal{A}(\pi, \xi) $ of $ \pi $ at $ \xi $ is defined as the union of the cluster sets along all the curves landing non-tangentially at $ \xi $. 
	\item The {\em (unrestricted) cluster set} $ Cl(\pi, \xi) $ of $ \pi $ at $ \xi $ is the set of values $ w\in\widehat{\mathbb{C}} $ for which there is a sequence $ \left\lbrace z_n\right\rbrace _n \subset\mathbb{D}$ such that $ z_n\to \xi $ and $\pi (z_n)\to w $, as $ n\to\infty $.
\end{itemize}

\begin{defi}{\bf (Boundary behaviour of $ \pi $, {\normalfont \cite[Def. 4.1]{FerreiraJove}}}\label{def:escaping-bounded-bungee})
	Let $ U $ be a multiply connected domain,  let $ \pi\colon\mathbb{D}\to U$ be a universal covering map and let $ \xi\in\partial\mathbb{D} $. Let $ R_\xi(t)=\left\lbrace t\xi \colon t\in \left[ 0,1\right)  \right\rbrace $.
	\begin{itemize}
		\item  We say that $ \xi\in\partial \mathbb{D} $ is {\em of escaping type} for $ \pi $ if $ \pi(R_\xi (t))\to\widehat{\partial} U $, as $ t\to 1^{-} $.
		\item We say that $ \xi\in\partial \mathbb{D} $ is {\em of bounded type} for $ \pi $ if $\left\lbrace  \pi(R_\xi (t))\colon t\in \left[ 0,1\right) \right\rbrace  $ is compactly contained in $ U $.
		\item Otherwise, we say that $ \xi\in\partial \mathbb{D} $ is {\em of bungee type} for $ \pi $.
	\end{itemize}
\end{defi}

\begin{prop}{\bf (Boundary behaviour of non-tangential curves, {\normalfont \cite[4.2, 4.15]{FerreiraJove}})}\label{prop-escaping-bounded-bungee}	Let $ U $ be a multiply connected domain,  let $ \pi\colon\mathbb{D}\to U $ be a universal covering map and let $ \xi\in\partial\mathbb{D} $. 
	\begin{itemize}
		\item  If $ \xi\in\partial \mathbb{D} $ is of {escaping} type and $ \eta $ is a curve landing non-tangentially at $ \xi$, then $ \pi(\eta (t))\to\widehat{\partial} U $ as $ t\to 1^{-} $, and
		\[ Cl_{\mathcal{A}}(\pi, \xi)= Cl_{R}(\pi, \xi)\subset \widehat{\partial} U. \] In particular, $ Cl_{\mathcal{A}}(\pi, \xi) $ is closed and contained in a unique boundary component of $U$. 
		
	\noindent	Conversely, if $ \eta $ is a curve landing  at $ \xi$ and $ \pi(\eta (t))\to\widehat{\partial} U $, as $ t\to 1^{-} $, then  $ \xi$ is of escaping type.
		\item If  $ \xi\in\partial \mathbb{D} $ is of bounded type and $ \eta $ is a curve landing non-tangentially at $ \xi$, then $\left\lbrace  \pi(\eta (t))\colon t\in \left[ 0,1\right) \right\rbrace  $ is compactly contained in $ U $. Moreover,
		\[ Cl_{\mathcal{A}}(\pi, \xi)=  U. \] In particular, $ Cl_{\mathcal{A}}(\pi, \xi) $ is open. 
		
	\noindent	Conversely, if $ \eta $ is a curve landing  at $ \xi$ and $\left\lbrace  \pi(\eta (t))\colon t\in \left[ 0,1\right) \right\rbrace  $ is compactly contained in $ U$, then  $ \xi$ is of bounded type.
		\item If $ \xi\in\partial \mathbb{D} $ is of {bungee} type and $ \eta $ is any curve landing at $ \xi $, then $\left\lbrace  \pi(\eta (t))\colon t\in \left[ 0,1\right) \right\rbrace  $ is neither compactly contained in $ U $ nor converging to $ \partial U $.
	\end{itemize}
\end{prop}

The proof of the previous result is essentially related to Ohtsuka's method of exhausting multiply connected domains, as we explain in \cite[Sect. 3.3]{FerreiraJove}. We give here a sketch of his construction.

Let $ U\subset\widehat{\mathbb{C}} $ be a multiply connected domain, which we assume to be of connectivity larger than two. Let $ \left\lbrace U_n\right\rbrace _{n\geq 1}$  be an exhaustion of $ U$ such that each $ U_n $ is a finitely connected domain compactly contained in $ U $, with $ \overline{U_n} \subset U_{n+1}$ for all $ n\geq 1 $, and $ \bigcup_n U_n =U$. Without loss of generality, we can assume that each $ U_n $ is bounded by (finitely many) simple closed analytic curves, each of them dividing $ U $ into two non-simply connected domains. With an appropriate labelling of the boundary components of $ U_n $ (which we shall not discuss here), the sequence $ \left\lbrace U_n\right\rbrace _{n\geq 1}$ is called an {\em Ohtsuka exhaustion} of $U $. By means of lifting the curves in $ \partial U_n $, one obtains an additional means of determining whether a point  is of bounded, bungee or escaping type \cite[Sect. 4.5]{FerreiraJove}.

Finally, we give a method for determining when the limit set of $ U$ is a Cantor set, by means of examining it in an elementary way (for details, see \cite[Sect. 5.2]{FerreiraJove}).
\begin{defi}{\bf (True crosscut)}\label{def-true-crosscut} A
	\textit{true crosscut} of $U$ is an open Jordan arc $C \subset U$ such that $\overline C = C\cup\{a, b\}$, with $a$ and $b$ distinct and contained on the same boundary component of $U$, and such that $U\setminus C$ has exactly two connected components, one of which is simply connected.
\end{defi}

\begin{prop}
	{\bf (Geometric characterization of limit sets, {\normalfont\cite[Corol. E]{FerreiraJove}})}\label{prop-geometri}
	Let $U\subset\Chat$ be a hyperbolic multiply connected domain, and let $\pi\colon\DD\to U$ be a universal covering map. Then, $\Lambda$ is a Cantor subset of $\partial\DD$ if and only if there exists a true crosscut in $U$.
\end{prop}

Note that, if there exists a true crosscut $C\subset U$, then it bounds a simply connected domain, say $U_C$. Then, $ \pi^{-1}(U_C) $ consists of pairwise disjoint (non-degenerate) crosscut neighbourhoods in $\mathbb{D}$ in which $\pi$ acts univalently (compare with \cite[Sect. 5.2]{FerreiraJove}). The set of points for which such crosscut neighbourhoods exist coincides with the set $\partial \mathbb{D}\smallsetminus\Lambda$.  
Note that true crosscuts lift to crosscuts in $\mathbb{D}$ (with the standard definition used in \cite[Chap. 17]{Milnor}, \cite[Chap. 2]{Pom92}). 

\subsection{Singular values}\label{subsection-SV} Let $ U ,V\subset\widehat{\mathbb{C}} $ domains, and $ f\colon U\to V $ holomorphic.
As usual, we say that $ v\in V$ is a {\em regular value} of $ f $  if there exists $ \rho>0 $ such that all inverse branches  of $ f $ are well-defined in $ D(v,\rho) $.
{\em Singular values} of $ f $, denoted by $ SV(f) $, are defined as the set of values in $ V $ which are not regular. 

In the case of inner functions, the following holds (see \cite[Sects. 2.6, 4]{Jov24} for details).

\begin{prop}{\bf (Singular values in $ \partial\mathbb{D} $,	{\normalfont\cite[Prop. 4.2]{Jov24}})}\label{prop-SVinner}
	Let $ g\colon\mathbb{D}\to\mathbb{D} $ be an inner function, and let $ \xi\in\partial\mathbb{D} $. The following are equivalent.
	\begin{enumerate}[label={\em(\alph*)}]
		\item There exists a crosscut $ C $, with crosscut neighbourhood $ N_C $, such that $ \xi\in\partial N_C $ and $ SV(g|_\mathbb{D})\cap N_C=\emptyset $.
		\item $ \xi $ is a regular value of $ g $.
	\end{enumerate}
\end{prop}

The following lemma exploits the relation between singular values of the map $ f\colon U\to V $ and the ones of the associated inner function.
\begin{lemma}{\bf (Singular values of the lift)}\label{lemma-sing-limit-sets}
	In the Setting \ref{setting}, assume that $U$ and $V$ are multiply connected. Then, the following hold.
	\begin{enumerate}[label={\em(\alph*)}]
		\item\label{lemma-singularities-a} $  SV(f|_{U})= \pi_V (SV(g)\cap\mathbb{D} )$.
		\item\label{lemma-singularities-b} Either $ SV(f|_{U}) = \emptyset$ and $ g $ is  a Möbius transformation, or $ SV(f|_{U}) \neq \emptyset$ and $ g $ is an inner function of infinite degree. 
		\item \label{lemma-singularities-c} If $ SV(f|_{U}) \neq \emptyset$, then $ \Lambda_U \subset E(g)$.
	\end{enumerate}
\end{lemma}
The hypothesis of multiple connectedness of $U$ and $V$ only comes into play in {\em \ref{lemma-singularities-b}}  to claim that $g$ has infinite degree. Indeed, if $U$ and $V$ are simply connected, then $ \pi_U $ and $ \pi_V$ are Riemann maps, so {\em \ref{lemma-singularities-c}} becomes trivial, and {\em \ref{lemma-singularities-a}} can be strengthened to $\pi_V^{-1}(SV(f|_U)) = SV(g)\cap \DD$.

\begin{proof}
	Statement {\em \ref{lemma-singularities-a}} follows immediately from the fact that $ \pi_U $ and $ \pi_V $ are local homeomorphisms. Indeed, let $ z\in V $ and consider a disk $ D(z,\rho) $, for some $ \rho>0 $, small enough so that all inverse branches of $ \pi_V  $ are well-defined in this disk. Then, by the commutative relation $ \pi_V\circ g=f\circ \pi_U $, it is clear that all inverse branches of $ f $ are well-defined in $ D(z,\rho) $ (so $ z $ is a regular value for $ f $) if and only if all inverse branches of $ g $ are well-defined in $ \pi_V^{-1}(D(z,\rho)) $ (note that this last set is a collection of disjoint disks), and thus each point in $ \pi_V^{-1}(z) $ is a regular value for $ g $. This already implies that $ z $ is a singular value for $ f $ if and only if there exists $w\in \pi_V^{-1}(z) $ which is a singular value for $ g $, so $  SV(f|_{U})= \pi_V (SV(g)\cap\mathbb{D} )$, as desired.

	For statement {\em \ref{lemma-singularities-b}}, note that, if $ SV(f|_{U}) = \emptyset$, then $ g $ is an inner function with no singular values, and thus a Möbius transformation. If $ SV(f|_{U}) \neq \emptyset$, then there exists at least one critical or asymptotic value inside $ U $. In the first case, if $ z $ is a critical value (i.e. $ f'(z)=0 $), then it follows from the commutative relation $ \pi_V\circ g=f\circ \pi_U $ that, for every $w\in \pi_V^{-1}(z) $, $ g'(w)=0 $ (note that $\pi'_U(z), \pi'_V(z)\neq 0$ for all $ z\in\mathbb{D} $). Thus, $ g $ has infinitely many critical values, so it has infinite degree. On the other hand, if $ z $ is an asymptotic value, it follows from the fact that $ \pi_U $ and $ \pi_V $ are local homeomorphisms that there is $ w\in \pi_V^{-1}(z) $ which is an asymptotic value for $ g $, and thus $ g $ has infinite degree.
	
	Finally, let us prove statement  {\em \ref{lemma-singularities-c}}. Proceeding as before, if $ SV(f|_{U}) \neq \emptyset$, then there exists either a critical or an asymptotic value in $ V $. In the first case, take $ z\in U $ to be a critical point. Then, every point in $ \pi_U^{-1}(z) $ is a critical point for $ g $. Since they accumulate at $ \Lambda_U $, it is clear that all points in $ \Lambda_U $ are singularities of $ g $. In the second case, one can take an asymptotic path $ \eta\colon \left[ 0,1\right) \to U $ which lands at a singularity $z_0\in E (f )$. By the Correspondence Theorem \cite[Thm. 3.4]{FerreiraJove}, since $ z_0 $ is a point on $ \widehat{\partial} U $ which is accessible from $ U $ through $ \eta $, $ \pi_U^{-1}(\eta(t)) $ lands at countably many points on $ \partial \mathbb{D} $, as $ t\to 1^- $. It is clear that such points are singularities of $ g $, and they accumulate at $ \Lambda_U $, implying that $ \Lambda_U \subset E(g)$, as desired.
\end{proof}

\begin{obs}
	In view of Lemma \ref{lemma-sing-limit-sets}{\em\ref{lemma-singularities-c}}, one might be tempted to think that $ \Lambda_V \subset SV(g)$. However, this is not true, as shows the example considered in Section \ref{subsect-baker}.
\end{obs}

\subsection{Mapping properties of the associated inner function}\label{subsect-mapping-properties}
We can also obtain information on how the boundary extension $g^*$ interacts with the limit sets. In the following lemma, we use the symbol $ \Subset $ to denote `compactly contained'. Recall that $ \Omega=\Omega(f)=\widehat{\mathbb{C}}\smallsetminus E(f)  $ is the largest set where $ f $ is meromorphic and $ E(f) $ is the set of singularities of $ f $.

\begin{lemma}{\bf (Lifts and boundary behaviour)}\label{lemma-mapping-properties}
	In the Setting \ref{setting}, the following holds.
	\begin{enumerate}[label={\normalfont(\alph*)}]
		\item {\em (Bounded type)}\label{lemma-mapping-properties-a}  If $ \xi  \in \partial \mathbb{D}$ is of bounded type for $ \pi_U $ and $ g^*(\xi) $ exists, then  $ g^*(\xi) $ is of bounded type for $ \pi_V $.
		
		\item  {\em (Escaping type I)} \label{lemma-mapping-properties-b}  If $ \xi  \in \partial \mathbb{D}$ is of escaping type for $ \pi_U $,  $g^*(\xi)$ exists and 
		\[Cl_\mathcal{A}(\pi_U, \xi)\Subset \Omega,\] then $ g^*(\xi) $  is of escaping type for $ \pi_V $.

				\item  {\em (Escaping type II)} \label{lemma-mapping-properties-c}  If $ \xi  \in \partial \mathbb{D}$ is of escaping type for $ \pi_U $, and there exists a component $ C $ of $ \partial U $ such that 
		\[Cl_\mathcal{A}(\pi_U, \xi)\subset C\Subset \Omega,\] then $ g^*(\xi) $ exists and it is of escaping type for $ \pi_V $. Moreover, $ \xi \in \Lambda_U$  if and only if  $ g^*(\xi) \in \Lambda_V$.
		
		\item  {\em (Bungee type)}\label{lemma-mapping-properties-d}  If $ \xi  \in \partial \mathbb{D}$ is of bungee type for $ \pi_U $, and there exists a component $ C $ of $ \partial U $ such that 
		\[Cl_\mathcal{A}(\pi_U, \xi)\cap C\neq\emptyset, \hspace{0.5cm} C\Subset \Omega;\] then $ g^*(\xi) $ exists and it is of bungee type for $ \pi_V $. 
	\end{enumerate}
\end{lemma}

\begin{proof}
	\begin{enumerate}[label={\normalfont(\alph*)}]
	\item   Let $ \xi\in\partial\mathbb{D} $ be of bounded type, and consider\[R_\xi(t)= \left\lbrace t\xi \colon t\in \left[ 0,1\right)  \right\rbrace  .\]
	By assumption $ \left\lbrace \pi_U(R_\xi(t))\colon t\in \left[ 0,1\right) \right\rbrace  $ is compactly contained in $ U $. Then, $$ g(R_\xi)\colon\left[ 0,1\right)\longrightarrow\mathbb{D} $$ is a curve landing at $ g^*(\xi)$,  and 
	\[\pi_V(g(R_\xi))=f(\pi_U(R_\xi))\colon\left[ 0,1\right)\longrightarrow V \] is compactly contained in $ V $ (since the image of a compact set is a compact set). By Proposition \ref{prop-escaping-bounded-bungee}, $ g^*(\xi) $ is of bounded type for $ \pi_V $.
	\item The proof in the escaping case is similar to the previous one. Indeed, let $\xi\in\partial\mathbb{D}$ be of escaping type, so $ R_\xi (t)\to \partial U \subset \Omega$ as $ t\to 1 ^-$. Then, $g(R_\xi)$ is a curve landing at $ g^*(\xi)$,  and 
	\[\pi_V(g(R_\xi))=f(\pi_U(R_\xi))\colon\left[ 0,1\right)\longrightarrow \widehat{\partial} V .\]
By Proposition \ref{prop-escaping-bounded-bungee}, $ g^*(\xi) $ is of escaping type for $ \pi_V $.
		\item Take a sequence of non-contractible simple curves $ \left\lbrace \sigma_n\right\rbrace _n $ isolating $ C $ from the other components of $ \widehat{\partial} U $ ($ \left\lbrace \sigma_n\right\rbrace _n $ can be thought as the boundary components of the elements of an Ohtsuka exhaustion). 
		
		Since $ C\Subset \Omega $, we assume without loss of generality that $ f $ is holomorphic in  the component of $ \widehat{\mathbb{C}}\smallsetminus \sigma_0 $ containing $C$. 
		By the homological version of the argument principle (see e.g. \cite[Thm. 18]{Ahl79}), $ \left\lbrace f(\sigma_n)\right\rbrace _n $ is a sequence of non-contractible curves (possibly non-simple) isolating $ f(C) $ from the other boundary components of $ \partial V $. Therefore, if $ \eta $ is a curve landing at $ \xi $ such that $ \pi_U(\eta) $ crosses all the curves $ \left\lbrace \sigma_n\right\rbrace _n $ to accumulate on $ C $, then $ f(\pi_U(\eta)) $ crosses all the curves $ \left\lbrace f(\sigma_n)\right\rbrace _n $ to accumulate on $ f(C) $.
	
		Thus, if $ \left\lbrace C_n\right\rbrace _n $ denotes all the crosscuts at $ \xi $ obtained by lifting the curves $ \left\lbrace \sigma_n\right\rbrace _n $, $ g(\eta) $ crosses all the crosscuts $ \left\lbrace g(C_n)\right\rbrace _n $ accumulating on $ \partial \mathbb{D} $. We have to see that $ g(\eta) $ actually lands at a point, say $ \zeta\in\partial\mathbb{D} $. Then, by the Lehto-Virtanen Theorem \cite[Sect. 4.1]{Pom92}, we will have $ g^*(\xi)=\zeta$ (see Figure \ref{fig-lemma3.9c}).
		
		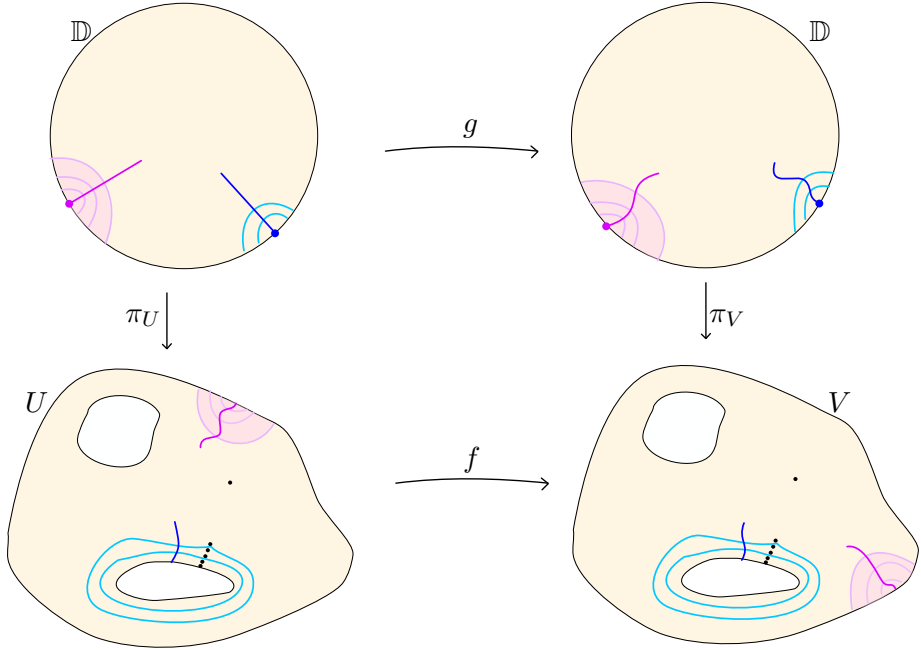
\begin{figure}[h]\centering
			\includegraphics[width=12cm]{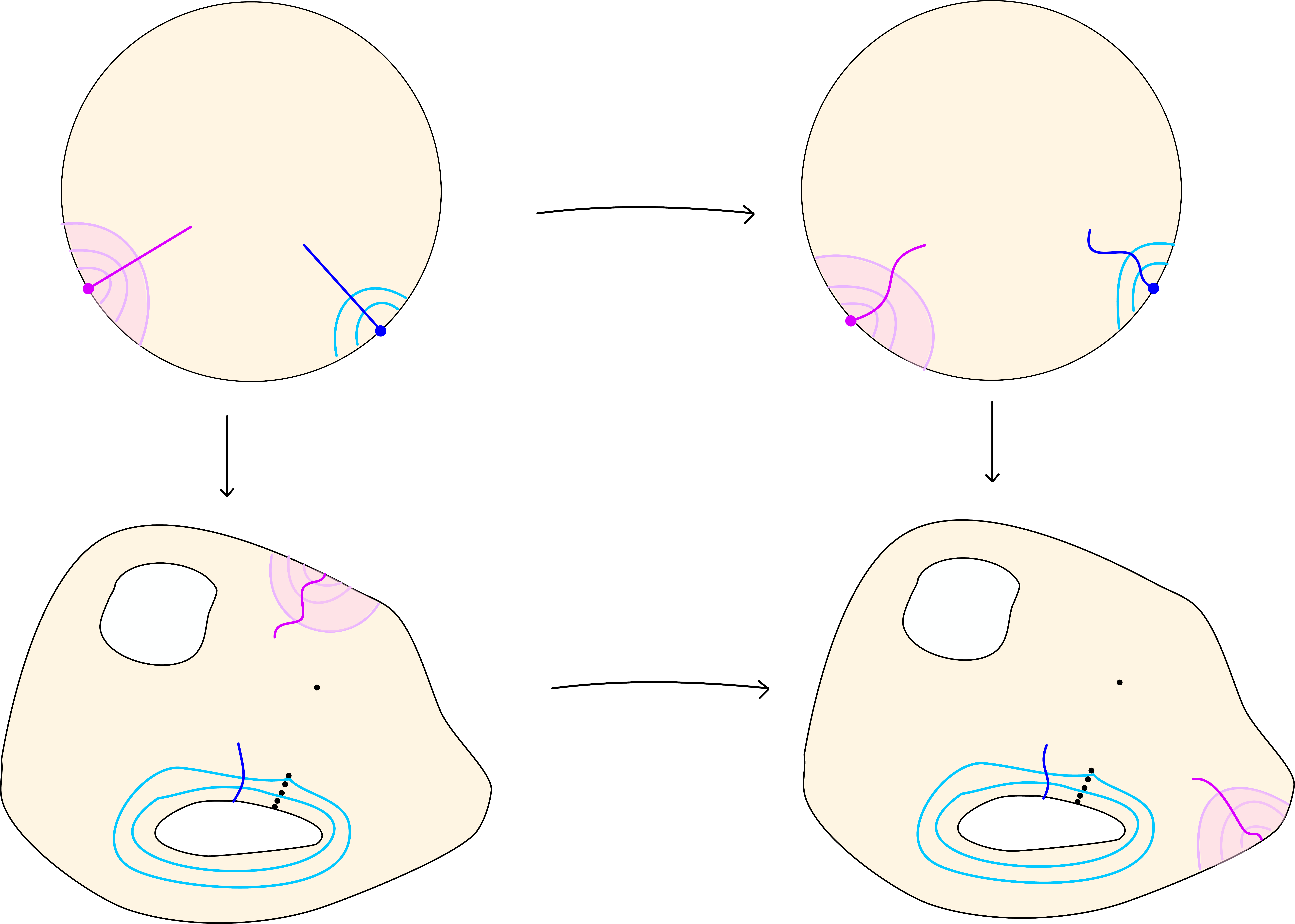}
			\setlength{\unitlength}{12cm}
			\put(-0.87, 0.36){$ \pi_U $}
				\put(-0.23, 0.36){$ \pi_V $}
							\put(-0.98, 0.26){$U $}
				\put(-0.1, 0.26){$ V $}
					\put(-0.93, 0.67){$\DD $}
				\put(-0.12, 0.67){$ \DD$}
					\put(-0.5, 0.57){$g $}
				\put(-0.5, 0.2){$ f$}
			\caption{\footnotesize Sketch of the construction in Lemma \ref{lemma-mapping-properties}\ref{lemma-mapping-properties-c}.}\label{fig-lemma3.9c}
		\end{figure}
		
		We proceed differently according to if $ \xi\in\partial\mathbb{D}\smallsetminus\Lambda_U $ or if $ \xi\in\Lambda_U $. In the first case,  we can find a crosscut neighbourhood $ N $ around $ \xi $ such that $ \pi_U(N) $ is simply connected and $ f|_{\pi_U(N)} $ is conformal (note that critical points are discrete in $ \Omega $, and $ C\Subset \Omega $). This implies that $ g|_N $ is conformal and thus $ \xi $ is not a singularity \cite[Thm. II.6.6]{garnett}, so $ g(\xi) $ is well-defined. By the commutative relation $ f\circ\pi_U=\pi_V\circ g $, $ g(N) $ is a crosscut neighbourhood of $ g(\xi) $ whose image under $ \pi_V $ is simply connected. Hence, $ g(\xi)\in\partial\mathbb{D}\smallsetminus\Lambda_V $.
		
		In the second case, if $ \xi\in\Lambda_U $, the curves $ \left\lbrace \sigma_n\right\rbrace _n $ are non-homotopic to each other, and thus the crosscuts $ \left\lbrace C_n\right\rbrace _n $ shrink to the point $ \xi $. By continuity, 
		the curves $ \left\lbrace f(\sigma_n)\right\rbrace _n $ are non-homotopic to each other, and thus the crosscuts $ \left\lbrace g(C_n)\right\rbrace _n $ shrink to the point $ \zeta $, with $ g^*(\xi)=\zeta  $.
		
		Note that the previous argument shows that $ \xi \in \Lambda_U$  if and only if  $ g^*(\xi) \in \Lambda_V$. 
		\item For the bungee type case, proceed as in {\em\ref{lemma-mapping-properties-c}}: take a sequence of non-contractible simple curves $ \left\lbrace \sigma_n\right\rbrace _n $ isolating $ C $ from the other components of $ \partial U $, and consider the sequence of crosscuts $ \left\lbrace C_m\right\rbrace _m $ arising from $ \pi_U^{-1}(\left\lbrace \sigma_n\right\rbrace _n) $ at $ \xi $. Since $ \xi $ is of bungee type, there exists  a subsequence $  \left\lbrace C_{m_j}\right\rbrace _j  $ such that $ \pi_U(C_{m_j})=\sigma_j$ and $ j\to\infty $, and another subsequence $  \left\lbrace C_{m_k}\right\rbrace _k  $ such that $ \pi_U(C_{m_k})= \sigma_{n_0} $ for some fixed $ n_0 $. By continuity (as in {\em\ref{lemma-mapping-properties-c}}), the same property holds for the crosscuts $ \left\lbrace g(C_n)\right\rbrace _n $. Therefore, any point lying in the circular interval delimited by $ \left\lbrace g(C_m)\right\rbrace _m $ is of bungee type. Since points of bungee type have zero $ \lambda $-measure (since $ \pi_V^* $ does not exist for such points), the circular interval delimited by $ \left\lbrace g(C_m)\right\rbrace _m $ is in fact degenerate and consists of a single point $ \zeta\in\partial\mathbb{D} $. Note that $ g(R_\xi) $ lands at $ \zeta $.
	\end{enumerate}
\end{proof}

In the case when  the singular values $SV(f|_U)$ are compactly contained within $U$, or $ f\colon U\to V $ is a proper map, the previous statement can be improved as follows. 

\begin{lemma}{\bf (Mapping properties for regular points)}\label{lemma-lifts-and-limit-sets-SV}
	In the Setting \ref{setting}, assume $SV(f|_U)$ is compactly contained in $U$. Let $\xi\in \partial \mathbb{D}\smallsetminus \Lambda_V$. 
	Then, for every $\zeta\in\partial\mathbb{D}$ such that $ g^*(\zeta)=\xi $, it holds that $ \zeta \notin E(g) $ and $ \zeta\in \partial \mathbb{D}\smallsetminus \Lambda_U$.
\end{lemma}
\begin{proof}
%
%
		
			Let $\xi\in \partial \mathbb{D}\smallsetminus \Lambda_V$. Since $SV(f|_U)\Subset U$, one can find a crosscut neighbourhood $ N $ of $ \xi $ such that $ \pi_V(N) $ is simply connected and $ \pi_V(N) \cap SV(f|_U)=\emptyset $. By Lemma \ref{lemma-sing-limit-sets}, there are no singular values of $ g $ in $N$. Applying Proposition \ref{prop-SVinner}, we get that there is an inverse branch of $g$, say $G_1$, such that $ G_1 (N)$ is a crosscut neighbourhood of $ \zeta $. Thus, $ g(G_1(N))=N $ and $\zeta$ is not a singularity \cite{garnett}. By the commutative relation $ \pi_V\circ g=f\circ \pi_U $, $\pi_U(G_1(N))$ is simply connected, implying that $\zeta\in\partial\mathbb{D}\smallsetminus\Lambda_U$, as desired.
\end{proof}

\begin{lemma}{\bf (Lifts and boundary behaviour for proper maps)}\label{lemma-lifts-and-limit-sets-proper}
	In the Setting \ref{setting}, assume $ f\colon U\to V $ is proper. Then, the following holds.
	\begin{enumerate}[label={\normalfont(\alph*)}]
		\item{\em (Escaping type)} If $ \xi  \in \partial \mathbb{D}$ is of escaping type for $ \pi_U $, then $ g^*(\xi) $ exists and it is of escaping type for $ \pi_V $. Moreover, $ \xi \in \Lambda_U$  if and only if  $ g^*(\xi) \in \Lambda_V$.
		\item{\em (Bungee type)} If $ \xi  \in \partial \mathbb{D}$ is of bungee type for $ \pi_U $, then $ g^*(\xi) $ exists and it is of bungee type for $ \pi_V $.
		\item{\em (Bounded type)} If $\xi\in \partial\mathbb{D} $ is such that $ g^*$ is not well-defined at $ \xi $, then $ \xi  $ is of bounded type for $ \pi_U$.
	\end{enumerate}
\end{lemma}
\begin{proof}
	We start with the following technical claim.
	
	\begin{claim*}
		There  is  an Ohtsuka exhaustion $ \left\lbrace U_n\right\rbrace _n $  of $ U $ such that $ \left\lbrace V_n=f(U_n)\right\rbrace _n $ is an Ohtsuka exhaustion of $ V $.
	\end{claim*}
	\begin{proof}
		Let $ \left\lbrace V_n\right\rbrace _n $ be an Ohtsuka exhaustion of $ V $. Then, by assumption, $ \left\lbrace V_n\right\rbrace _n $ satisfies that each $ V_n $ is finitely connected, $ \overline{V_n}\subset V_{n+1} $, and $ \cup_n V_n =V $. Let $U_n = f^{-1}(V_n)\cap U$; we can take $V_1$ large enough that each $U_n$ is connected. We shall see that $ \left\lbrace U_n\right\rbrace _n  $ also satisfies the properties of being an Ohtsuka exhaustion, but this comes straightforward from the fact that $ f\colon U\to V $ is proper. Then, $ \left\lbrace U_n\right\rbrace _n  $ is the Ohtsuka exhaustion of $ U $ we were looking for.
	Note that the assumption that $ f\colon U\to V $ is proper is essential: otherwise $ \overline{U_n}\subset U_{n+1} $ may not hold.
	\end{proof}
	
			\begin{figure}[h]\centering
		\includegraphics[width=12cm]{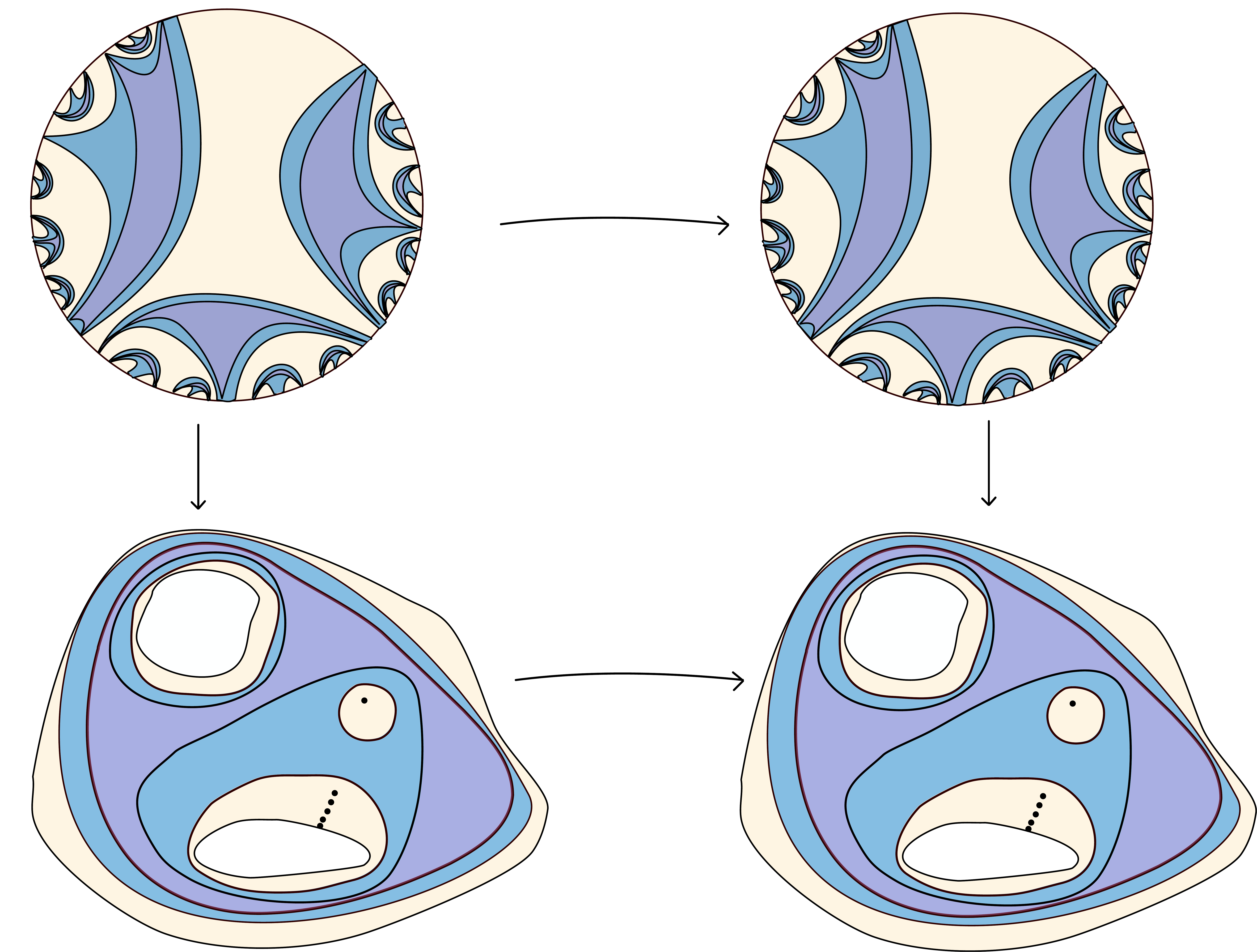}
		\setlength{\unitlength}{12cm}
			\put(-0.89, 0.38){$ \pi_U $}
		\put(-0.21, 0.38){$ \pi_V $}
		\put(-0.98, 0.26){$U $}
		\put(-0.08, 0.26){$ V $}
		\put(-0.98, 0.67){$\DD $}
		\put(-0.1, 0.67){$ \DD$}
		\put(-0.5, 0.6){$g $}
		\put(-0.5, 0.23){$ f$}
		\caption{\footnotesize Sketch of the construction in the claim of Lemma \ref{lemma-lifts-and-limit-sets-proper}. }\label{Ohtsuka}
	\end{figure}
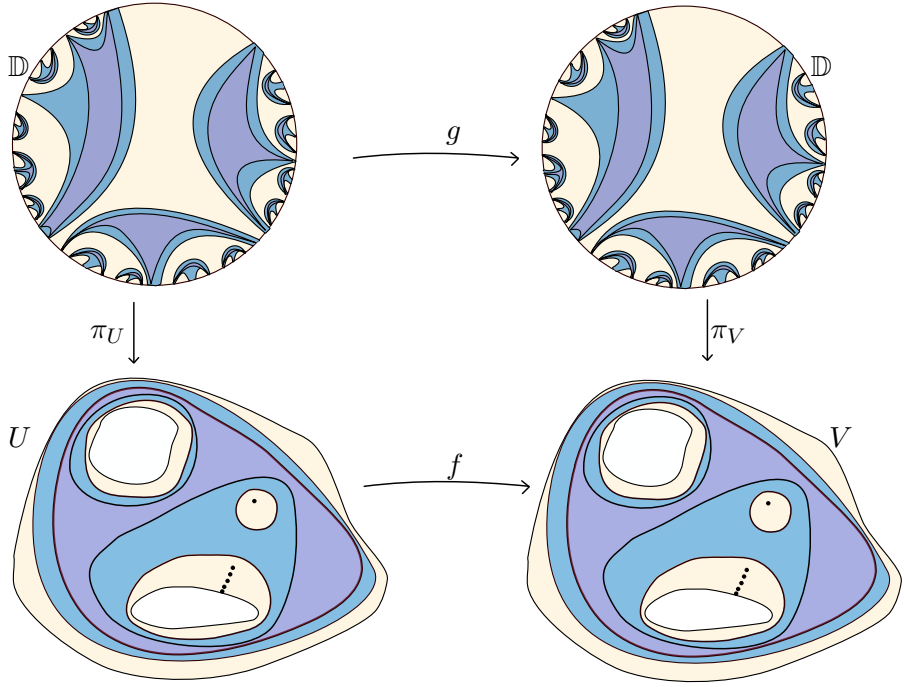

	Now, let us prove the statements in Lemma \ref{lemma-lifts-and-limit-sets-proper}. For the remainder of the proof, we let $ \left\lbrace U_n\right\rbrace _n $ be the Ohtsuka exhaustion of $ U $ constructed in the previous claim.
	\begin{enumerate}[label={\normalfont(\alph*)}]
		\item Let $ \xi\in\partial\mathbb{D} $ be of escaping type. Then, consider all the crosscuts $ \left\lbrace C_m\right\rbrace _m $ out of all the curves in $ \left\lbrace \pi_U^{-1}(\partial U_n) \right\rbrace _n $ separating $ 0 $ from $ \xi $, and order them in such a way that $ C_m $ separates $ 0 $ from $ C_{m+1} $ (this is a {\em fundamental sequence of crosscuts}, in the terminology of \cite[Def. 4.5]{FerreiraJove}). Take $ \eta $ to be a curve landing at $ \xi $, and crossing each of the $ C_m $'s once. Then, $ \pi_U(\eta) $ is a curve converging to $ \partial U $ in such a way that crosses each $ \pi_U(C_m) \in \partial U_n$ once.
		
		By means of the commutative relation $ f\circ\pi_U=\pi_V\circ g $, we deduce that $ g(\eta) $ is a curve converging to $ \partial \mathbb{D} $, crossing precisely once each of  the crosscuts $ \left\lbrace g(C_m)\right\rbrace _m $ in $ \left\lbrace \pi_V^{-1}(\partial V_n) \right\rbrace _n $. We claim that $ g(\eta) $ lands at a point $ \zeta $ in $ \partial \mathbb{D} $ (and thus $ g^*(\xi)=\zeta $, by the Lehto-Virtanen Theorem  \cite[Sect. 4.1]{Pom92}, and $ \zeta $ is of escaping type, by Prop. \ref{prop-escaping-bounded-bungee}). Assume, by contradiction, that $ g(\eta) $ accumulates on a continuum $ I $ on $ \partial \mathbb{D}$. In such a case, this would mean that the crosscuts $ \left\lbrace g(C_m)\right\rbrace _m $ do not shrink to a point, and hence $ I\subset \partial\mathbb{D}\smallsetminus\Lambda_V $. In particular, there exists a crosscut neighbourhood $ N $ of $ I $  in which $ \pi_V $ acts conformally, and hence $ \pi_V(N) $ is simply connected. Since $ f\colon U\to V $ is proper (and hence it has a finite number of singular values, all of them critical), we can take $ N $ small enough so that there are no singular values in $ \pi_V(N) $. In particular, this would imply that all inverse branches of $ g $ are well-defined and conformal in a neighbourhood of $ N $ (Prop. \ref{prop-SVinner}). This contradicts the fact that $ g(\eta) $ accumulates in a continuum (since we assumed that $ \eta $ lands at a point). Therefore, $ g(\eta) $ lands at a point $ \zeta $ in $ \partial \mathbb{D} $, with $ g^*(\xi)=\zeta $.
		
		It is left to see that $ \xi \in \Lambda_U$  if and only if  $\zeta = g^*(\xi) \in \Lambda_V$. The argument above shows that $ \zeta\notin \Lambda_V $ implies $ \xi\notin \Lambda_U $. To prove the reciprocal, let $ \xi \notin \Lambda_U$, then there exists a crosscut neighbourhood $ N $ of $ \xi $ such that $ \pi_U|_N $  is conformal. Taking $ N $ smaller (but non-degenerate), one can assume that $ g|_N $ is conformal (and so is $ f|_{\pi_U(N)} $). Therefore, $ g(N) $ is a crosscut neighbourhood of $ \zeta $ in which $ \pi_V $ acts conformally. This implies $ \zeta\notin \Lambda_V $, and ends the proof of the statement.
		\item The proof proceeds as in the escaping type case. Indeed, let $ \xi\in\partial\mathbb{D} $ be of bungee type, and let $ \left\lbrace C_m\right\rbrace _m $ be the fundamental sequence of crosscuts for $ \xi $, with respect to the exhaustion $ \left\lbrace U_n\right\rbrace _n $. Let $ \eta$   be a curve landing at $ \xi $, and crossing each of the $ C_m $'s once.
		Then, $ g(\eta) $ is a curve landing at  $ \partial \mathbb{D} $, crossing precisely once the crosscuts $ \left\lbrace g(C_m)\right\rbrace _m $ in $ \left\lbrace \pi_V^{-1}(\partial V_n) \right\rbrace _n $. Hence $ \left\lbrace g(C_m)\right\rbrace _m $ is a sequence of nested crosscuts. 
		
		We shall see that $ \left\lbrace g(C_m)\right\rbrace _m $ shrinks to a point.
		To that end, note that the sequence of crosscuts $ \left\lbrace g(C_m)\right\rbrace _m $ contains a subsequence $  \left\lbrace g(C_{m_j})\right\rbrace _j  $ such that $ g(C_{m_j})\in \partial V_{n_j} $ and $ m_j\to\infty $, and another subsequence $  \left\lbrace g(C_{m_k})\right\rbrace _k  $ such that $ g(C_{m_k})\in \partial V_{n_0} $ for some fixed $ n_0 $ (because the same was true for $ \left\lbrace C_m\right\rbrace _m $ since $ \xi $ is of bungee type, and $ U_n=f^{-1}(V_n) $, for all $ n\geq 0 $). Therefore, any point lying in the circular interval delimited by $ \left\lbrace g(C_m)\right\rbrace _m $ is of bungee type. Since points of bungee type have zero $ \lambda $-measure (since $ \pi_V^* $ does not exist for such points), the circular interval delimited by $ \left\lbrace g(C_m)\right\rbrace _m $ is in fact degenerate and consists of a single point $ \zeta\in\partial\mathbb{D} $. 
		
		Since  $ g(\eta) $ crosses precisely once the crosscuts $ \left\lbrace g(C_m)\right\rbrace _m $, $ g(\eta) $ lands at $ \zeta $.
		As a consequence of the Lehto-Virtanen Theorem  \cite[Sect. 4.1]{Pom92}, we have that $ g^*(\xi) =\zeta$, since $ \eta $ lands at $ \xi $ and $ g(\eta) $ lands at $ \zeta $, and $ \zeta $ is of bungee type, as desired.
		
		\item It is a reformulation of the previous two statements. Indeed, if $ \xi\in\partial\mathbb{D} $ is of escaping or bungee type, the radial limit exists. Hence, if the radial limit does not exist, $ \xi $ must be of bounded type.
	\end{enumerate}
\end{proof}

\begin{obs}
	Lemmas \ref{lemma-mapping-properties} and \ref{lemma-lifts-and-limit-sets-proper} are written in the terminology of points of escaping, bungee and bounded type, as introduced in \cite[Def. 4.1]{FerreiraJove}. In \cite[Sect. 5.2]{FerreiraJove}, the authors introduced a refined classification of points of escaping type in terms of a classification of prime ends. It seems plausible that this classification of prime ends is stable when considering maps $ f\colon U\to V $, so one may obtain results analogous to the ones in 	Lemmas \ref{lemma-mapping-properties} and \ref{lemma-lifts-and-limit-sets-proper}. Note, however, that we do not need this refined approach to work with Fatou components in the subsequent sections.
\end{obs}

\begin{obs}
Fixed points of parabolic deck transformations have been excluded from the previous discussion since Fatou components cannot have isolated boundary points (recall that the Julia set is perfect, Thm. \ref{thm-dynamicsK}{\em \ref{Fatou3}}). In a more general set-up, one can consider $ f\colon U\to V $ such that $ U $ has isolated boundary points. It seems plausible to extend the previous lemmas to fixed points of parabolic deck transformations.
\end{obs}

\subsection{Invariance of limit sets in the sense of measure} 
From the results in Section \ref{subsect-mapping-properties}, we see that regular points are mapped to regular points and singular points are mapped to singular points, as long as there is some control on $ E(f) $ and $SV(f|_U)$.  In fact, in the particular case when $ f\colon U\to V $ is proper, the following is an immediate consequence of Lemma \ref{lemma-lifts-and-limit-sets-proper}.

\begin{cor}{\bf (Invariance of limit sets for proper maps)}\label{cor-invariance-limit-sets-proper}
In the Setting \ref{setting}, assume $ f\colon U\to V $ is proper. Then, \[\Lambda_U=(g^*)^{-1}(\Lambda_V). \]
\end{cor}

In the more general setting where it is only assumed that $ SV(f|_U)$ is compactly contained in $U$, we can still guarantee that there is some sort of invariance in the sense of measure, as the following proposition shows.
\begin{prop}{\bf (Invariance of limit sets)}\label{prop-invariance-limit-sets}
	In the Setting \ref{setting}, assume $ SV(f|_U)$ is compactly contained in $U$. Then,  there exists a subset $ Z\subset \partial\mathbb{D}\smallsetminus\Lambda_V $ such that
	\[ (g^* )^{-1}(Z)\subset \partial\mathbb{D}\smallsetminus\Lambda_U \ \  \textrm{ and } \ \  \lambda (Z)=\lambda (\partial\mathbb{D}\smallsetminus\Lambda_V).\]
	
\end{prop}
\begin{proof}
	Let \[X\coloneqq \left\lbrace \xi\in\partial \mathbb{D} \colon \pi_U^*(\xi)\textrm{ exists},\ \pi_U^*(\xi)\notin E(f), \ f'(\pi_U^*(\xi))\neq 0\right\rbrace \subset \partial\mathbb{D}. \] That is, $ Y\coloneqq \partial\mathbb{D}\smallsetminus X $ consists of all points $ \xi\in\partial\mathbb{D} $ for which either the radial limit does not exist or it exists, but it is a singularity or a critical point. It is clear that $ \lambda (Y)=0 $. We claim that  $ \lambda(g^*(Y)) =0$. Indeed, assume, on the contrary, that $ g^*(Y) $ has positive $ \lambda $-measure. This would mean, by Lemma \ref{lemma-radial-limits}, that there exists a set of positive measure of preimages under $ \pi_V $ of essential singularities or critical values of $ f $. This is  a contradiction because the latter sets are countable (and hence their preimages under $ \pi_V $ have zero measure). Hence we deduce that $ \lambda(g^*(Y))=0 $,  as desired.

	Let\[Z\coloneqq \left\lbrace \xi\in\partial\mathbb{D}\colon (g^*)^{-1}(\xi)\subset X\right\rbrace \cap{(\partial \mathbb{D}\smallsetminus \Lambda_V)}\subset\partial\mathbb{D}.\]
	The set $\left\lbrace \xi\in\partial\mathbb{D}\colon (g^*)^{-1}(\xi)\subset X\right\rbrace$ has full measure, and so by construction $ \lambda(Z)=\lambda(\partial\mathbb{D}\smallsetminus\Lambda_V) $. 
	\begin{claim*}
		If   $ \xi\in Z $, then  $ (g^{*})^{-1}(\xi) \subset  \partial\mathbb{D}\smallsetminus\Lambda_U$.
	\end{claim*}
	\begin{proof}
	Let $ \zeta\in (g^{*})^{-1}(\xi) $. Since $ \pi_V^*(\xi)$ is not a singular value for $ f$ (otherwise for some $ \eta\in  (g^{*})^{-1}(\xi)$, $ \pi_U^*(\eta) $ is either a critical point or a singularity of $ f $, a contradiction), then Lemma \ref{lemma-lifts-and-limit-sets-SV} applies to show that $ \zeta \in\partial\mathbb{D}\smallsetminus\Lambda_U$. 
	\end{proof}
Hence, the statement in the proposition follows now from the previous claim.

\end{proof}

%
%
\section{Multiply connected basins of rational maps. Theorem \ref{thmB}}
Now we shall prove Theorem \ref{thmB}, which asserts that, for an invariant multiply connected basin $U$ of a rational map, the following hold.
\begin{enumerate}[label={(\alph*)}]
	\item The limit set $ \Lambda$ of the universal covering of $U$  is  $\partial\mathbb{D} $.
	\item The associated inner function $ g $ has infinite degree, and $ E(g) =\partial\mathbb{D}$.
	\item $ (g^*)^n(\xi) $ is well-defined  for all $ n\geq 0 $ if and only if $ \left\lbrace \pi(t\xi)\colon\xi\in \left[ 0,1\right) \right\rbrace  $ is not compactly contained within $ U $ (or, equivalently, if and only if $ \xi $ is not of bounded type).
\end{enumerate}

The main tool to prove the previous statements is the mapping properties developed in Section \ref{section-main-construction} for proper maps between multiply connected domains. Hence, we shall work in the following more general setting.

\begin{setting}\label{settingProper} Let $ f \in\mathbb{K}$, and let $ U$ be an attracting or parabolic basin. Assume $ U $ is multiply connected, and $ f|_U $ is proper.
	Consider  $\pi\coloneqq \pi_U\colon\DD\to U $ a universal covering map, and let $ g\colon\mathbb{D}\to\mathbb{D} $ be such that $ \pi \circ g= f\circ \pi$.
\end{setting}

It is clear that proving Theorem \ref{thmB} under the assumptions given in \ref{settingProper} implies the theorem in its original form. We prove it next, but first we discuss which inner functions can be associated with invariant Fatou components.
\subsection{Inner functions associated with invariant Fatou components} The following proposition gives a complete description of the inner functions which can be associated to invariant Fatou components, depending on the type of Fatou component. Note that this proposition, in particular, gives strong constraints on which functions can arise as associated inner functions for rational maps. 
\begin{prop}{\bf (Inner functions associated with invariant Fatou components)}\label{prop-inner-rational}
	Let $ f \in \mathbb{K}$, and let $ U $ be an invariant Fatou component of $ f $, such that  $U$ is not a Baker domain. Let $ \pi\colon\mathbb{D}\to U $ be a universal covering map, and let $ g\colon\mathbb{D}\to\mathbb{D} $ be an inner function such that $ \pi\circ g=f\circ \pi $. Then, the following holds.
	\begin{enumerate}[label={\em (\alph*)}]
		\item $ g $ is conjugate to an irrational rotation if and only if $ U $ is a Siegel disk.
		\item $ g $ is a hyperbolic Möbius transformation if and only if $ U $ is a Herman ring.
		\item $ g $ is a non-univalent elliptic inner function if and only if $ U $ is a basin of attraction. 
		\item $ g $ is a doubly parabolic inner function  if and only if $ U $ is a parabolic basin. The inner function $ g $ associated to a parabolic basin satisfies that $ g^* $ is recurrent with respect to the Lebesgue measure.
	\end{enumerate}
Moreover, assume $U$ is a basin, and $f|_U$ is a proper map. Then, $ g $ has finite degree if and only if $U$ is simply connected; otherwise $g$ has infinite degree. 
\end{prop}
\begin{proof}[Proof of Proposition \ref{prop-inner-rational}]
	According to Lemma \ref{lemma-sing-limit-sets}, $ f|_U $ has no critical points if and only if $ g $ is a Möbius transformation. Hence, $g$ is a disk automorphism if and only if $U$ is either a Siegel disk or a Herman ring (since basins always contain a singular value). 
	
	By definition, if $U$ is a Siegel disk, then $U$ is simply connected and $ f|_U $ is conjugate to  an irrational rotation.  We are left to check that the Möbius transformation associated to a Herman ring must be hyperbolic. By definition, if $U$ is a Herman ring, then $ U $ is doubly connected and can be uniformized to a non-degenerated round annulus. Thus, the group of deck transformations of $\pi$ has a  unique (hyperbolic) generator, and the limit set $\Lambda$ consists of two points. 
  Since $f|_U$  is conjugate to  an irrational rotation, $ f|_U $ preserves the (unique) hyperbolic geodesic $\gamma\subset U$, and so $g$ must preserve the geodesic $\tilde\gamma = \pi^{-1}(\gamma)$, which is a curve in $\mathbb{D}$ joining the two points in the limit set. Thus $g$ has two fixed points on $\partial \mathbb{D}$, so it is hyperbolic (compare also with the explanation in Sect. \ref{section-herman-ring}).

In (c) and (d), it is now clear that $g$ must be a non-univalent inner function. By Lemma \ref{lem:fp} (to be proved below), $g$ has a fixed point in $\DD$ if and only if $f|_U$ has a fixed point, and is therefore a basin of attraction. Inner functions associated with parabolic basins are always doubly parabolic, and their radial extension is always recurrent \cite[Thm. I]{DM91}. 

We are left to prove the claim about the degree of $g$ when $ f|_U $ is proper. If $U$ is multiply connected, then given any critical point $c\in U$ (which exists, since $U$ is an attracting or parabolic basin), any of the infinitely many points in $\pi^{-1}(c)$ is a critical point of $g$, and so $g$ has infinite degree. The converse is also clear, since if $U$ is simply connected then $g$ and $f|_U$ are conjugate by the Riemann map, and so have the same degree ($f|_U$ is proper, so it has necessarily finite degree).
\end{proof}

We finish this section with the following promised lemma.\footnote{The following result may be well-known (see \cite[p. 23]{DM91}), but we include its proof since we were unable to find an explicit reference for it.}

\begin{lemma}{\bf (Fixed points)}\label{lem:fp}
	Let $ U\subset\Chat $ be a hyperbolic domain, and let $ f\colon U\to U $ holomorphic. Consider $ \pi\colon\DD\to U $ a universal covering map, and let $ g\colon\mathbb{D}\to\mathbb{D} $ be such that $ \pi \circ g= f\circ \pi$. Then, 
	$ f|_U $ has a fixed point if and only if $ g|_\DD $ has a fixed point.
\end{lemma}
\begin{proof}
	The right-to-left implication is immediate. Indeed, assume $ g|_\DD $ has a fixed point, say $ p $. Then, $$ \pi(p) =\pi(g(p))=f(\pi(p)),$$ so $z= \pi(p)\in U $ is a fixed point for $ f|_U $.
	
	For the left-to-right implication, assume $ f|_U $ has a fixed point, say $ z \in U$. Let us prove that there exists $ q\in\pi^{-1}(z) $ such that $ g(q)=q $. Assume, on the contrary, that there is no fixed point for $ g $ in $ \pi^{-1}(z) $. In this case, $ g|_\DD $ has no fixed point at all (if it had another, $f|_U$ would have two distinct fixed points, which is a contradiction). By the Denjoy-Wolff Theorem, there exists $ p\in\partial\DD $ (the Denjoy-Wolff point) towards which orbits converge uniformly on compact subsets of $ \DD $.
	
	Now, take $ r >0$ small enough so that the disk $ D(z,r)\subset U $ is forward invariant under $ f $. Such a disk is {\em absorbing} for $ f|_U $, i.e. for every compact set $ K\subset U $, there exists $ n \geq 0$ such that $ f^n(K)\subset D(z,r) $. Note that $ \pi^{-1}(D(z,r)) $ consists of countably many topological disks in $ \DD $, each of them compactly contained in $ \DD $. It follows from the commutative relation $ \pi\circ g=f\circ \pi $ that such disks accumulate at the Denjoy-Wolff point $ p\in\partial \DD $, and none of the disks is forward invariant under $ g $. Thus, take $w\in \pi^{-1}(D(z,r)) $ and $ n\geq 0 $ so that $ g^n(w) $ lies in another connected component of $ \pi^{-1}(D(z,r)) $. Join these two points by a curve $ \gamma $, and consider the curves \[\Gamma_k\coloneqq \bigcup\limits_{N>k}g^{n\cdot N}(\gamma),\] which is a family of forward invariant curves landing at $ p $. Note that  $\Gamma_k\not\subset \pi^{-1}(D(z,r)) $. Since $ \pi\circ g=f\circ \pi $, \[\pi(\Gamma_k)= \bigcup\limits_{N>k}\pi(g^{n\cdot N}(\gamma))=\bigcup\limits_{N>k}f^{n\cdot N}(\pi(\gamma))\not\subset D(z,r)\]
	Thus, $ \pi(\gamma)\subset U $ is compact, but there does not exist $ n\geq 0 $ such that $ f^n(\pi(\gamma))\subset D(z,r) $. 
	This is a contradiction, implying that there exists $ q\in\pi^{-1}(z) $ such that $ g(q)=q $, as desired. See Figure \ref{fig-fixedpoints}.
\end{proof}

	\begin{figure}[h]\centering
	\includegraphics[width=15cm]{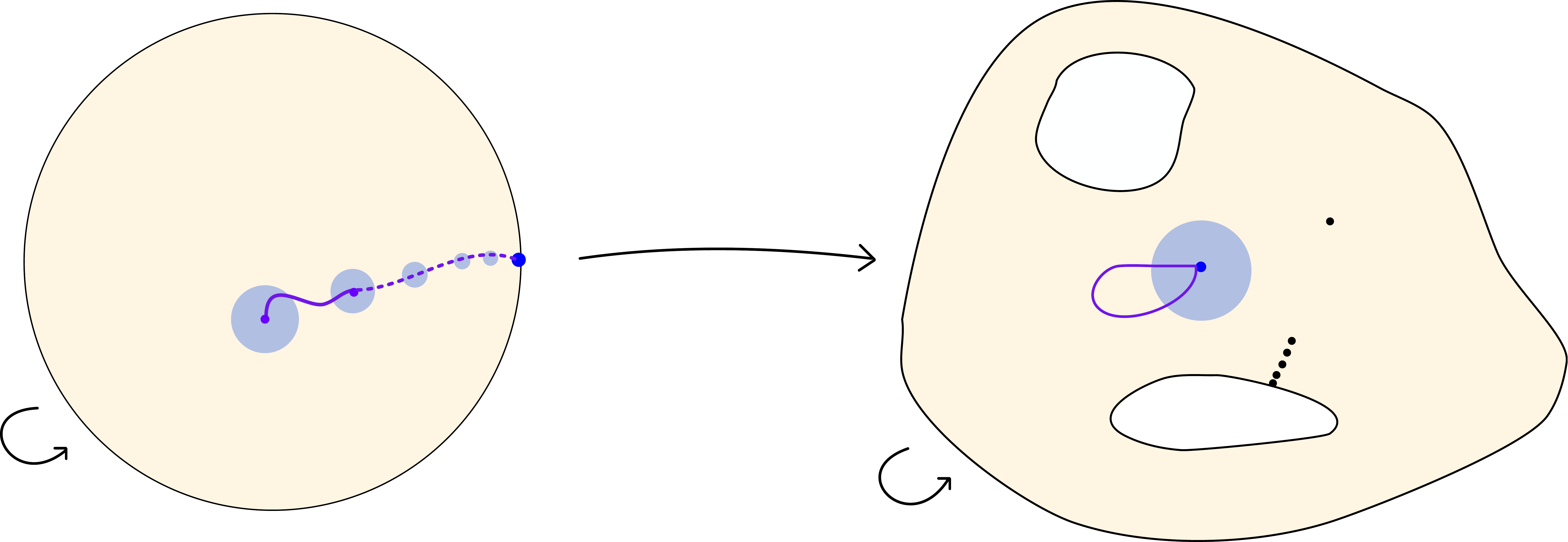}
	\setlength{\unitlength}{15cm}
	\put(-0.55,0.195){$ \pi $}
		\put(-0.66,0.18){$ p$}
				\put(-0.81,0.16){$ \gamma$}
								\put(-0.355,0.16){$ \pi(\gamma)$}
			\put(-0.23,0.18){$ z $}
		\put(-0.46,0.03){$ f $}
			\put(-1,0.03){$ g $}
				\put(-0.11,0.3){$ U $}
			\put(-0.97,0.3){$ \DD $}
	\caption{\footnotesize Sketch of the
		 construction in Lemma \ref{lem:fp}. Following the notation of the Lemma, $ z\in U $ represents the fixed point in $U$, and the blue disk around it, its absorbing domain. In the unit disk, the lift of such absorbing domain, accumulating at the Denjoy-Wolff point $p\in\partial\mathbb{D}$. }\label{fig-fixedpoints}
\end{figure}
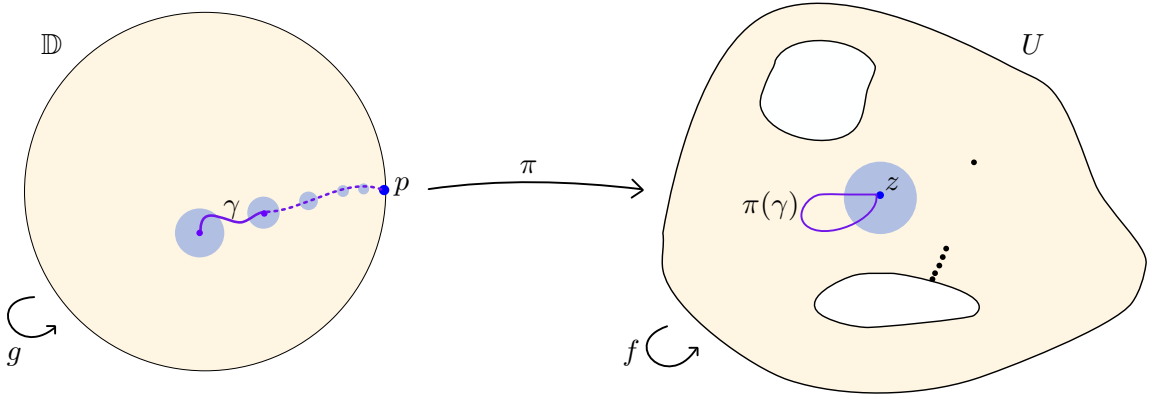

\subsection{Proof of Theorem \ref{thmB}} We are working under the assumptions of Setting \ref{settingProper}.  
\begin{enumerate}[label={(\alph*)}]
	\item We have to see that the limit set  $ \Lambda$ is  $\partial\mathbb{D} $. According to Corollary \ref{cor-invariance-limit-sets-proper}, $ (g^*)^{-1}(\Lambda)=\Lambda $. By Proposition \ref{prop-inner-rational} combined with Theorem \ref{thm-ergodic-properties}, $ g^* $ is ergodic with respect to the Lebesgue measure $\lambda$, so $ \lambda (\Lambda)=0 $ (and hence $ \Lambda $ is a Cantor set of zero measure), or $ \lambda (\Lambda)=1 $ (and hence $ \Lambda =\partial\mathbb{D}$). We have to see that the first case cannot happen. Assume to the contrary that $ \Lambda $ is a Cantor set, and let $I$ be an open circular interval in $ \partial\mathbb{D}\smallsetminus \Lambda $. By Lemma \ref{lemma-elliptic-inner}, $ (g^*)^n (I) =\partial\mathbb{D}$ for some $n\geq 0$. This would imply that some points in $ \partial\mathbb{D}\smallsetminus \Lambda $ are mapped to $ \Lambda $ -- a contradiction.
	\item By Proposition \ref{prop-inner-rational}, $ g $ has infinite degree. By Lemma \ref{lemma-sing-limit-sets}, $\Lambda\subset E(g)$, thus $ E(g) =\partial\mathbb{D}$.
	\item Let us see first that, if  $ \left\lbrace \pi(t\xi)\colon\xi\in \left[ 0,1\right) \right\rbrace  $ is not compactly contained within $ U $, then $ (g^*)^n(\xi) $ is well-defined  for all $ n\geq 0 $. Indeed, if  $ \left\lbrace \pi(t\xi)\colon\xi\in \left[ 0,1\right) \right\rbrace  $ is not compactly contained in $ U $, $ \xi $ is either of escaping or  bungee type. By Lemma \ref{lemma-lifts-and-limit-sets-proper}, $g^*(\xi)$ is well-defined, and it is again  of escaping or  bungee type. Proceeding by induction, we see that $ (g^*)^n(\xi) $ is well-defined  for all $ n\geq 0 $.
	
	For the converse, assume $ \left\lbrace \pi(t\xi)\colon\xi\in \left[ 0,1\right) \right\rbrace  $ is  compactly contained within $ U $, and let us see that $ (g^*)^n(\xi) $ is eventually not well-defined.  If $ \xi$ is of bounded type, then $ \xi  $ belongs to the non-tangential limit set of some finitely generated subgroup of the group of  deck transformations $ \Gamma $ \cite[Summary 4.15]{FerreiraJove}, whose generators correspond to the closed loops $ \sigma_1, \dots, \sigma_m $. 
	
	Assume on the contrary, that $ (g^*)^n (\xi)$ exists for all $ n\geq0 $. By the commutative relation $ f\circ \pi=\pi\circ g $, $ (g^*)^n (\xi)$ belongs to the non-tangential limit set of the subgroup generated by $ f^n(\sigma_1), \dots, f^n(\sigma_m) $. However, since for these Fatou components there exists a simply connected absorbing domain \cite{BFJK_AbsorbingSets}, there exists $ N\geq 1 $ such that $ f^N(\sigma_1), \dots, f^N(\sigma_m) $ are contained in the absorbing domain, and hence contractible. The subgroup generated by them is, therefore, trivial, and has an empty limit set -- a contradiction.
\end{enumerate}
Thus, Theorem \ref{thmB} is proved. \hfill $\square$

\section{Transcendental functions with pathological boundaries}\label{section-pathological-boundaries}
We consider here three examples of doubly connected Fatou components for functions in class $\mathbb{K}$. In contrast with rational maps, doubly connected Fatou components may not be Herman rings, and hence exhibit different boundary dynamics, which may be considered as pathological in view of Theorem \ref{thmB}. We discuss these pathological behaviours next.

\subsection{Herman rings}\label{section-herman-ring} Herman rings are, by definition, doubly connected invariant Fatou components in which $f$ acts as an irrational rotation.  Inner functions associated with Herman rings  are always hyperbolic Möbius transformations (Prop. \ref{prop-inner-rational}, compare with Fig. \ref{fig-HermanRing}).

\begin{figure}[h]\centering
	\includegraphics[width=8cm]{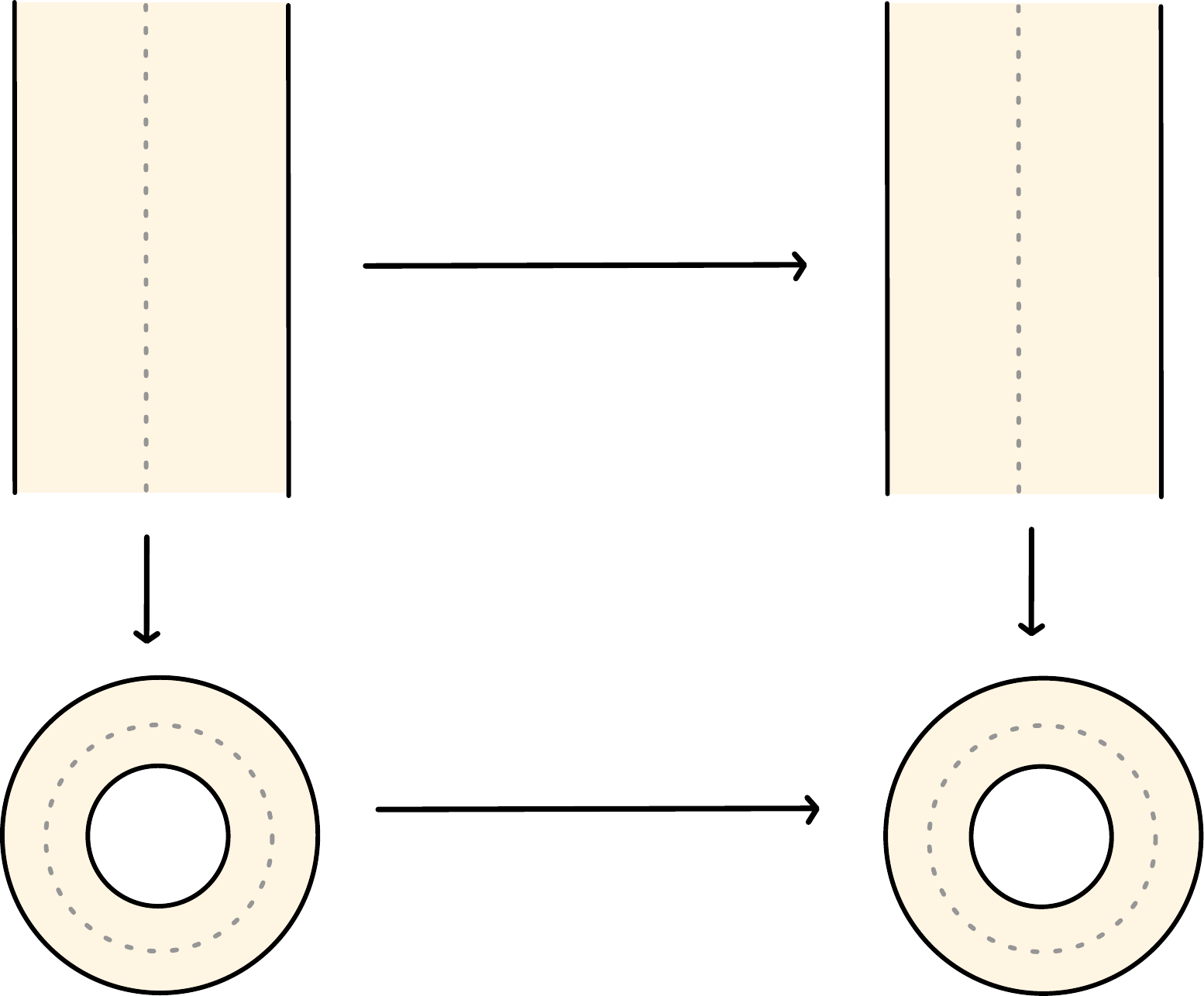
	}
	\setlength{\unitlength}{8cm}
	\put(-1.03, 0.77){$ S $}
		\put(-0.02, 0.77){$ S $}
			\put(-1.03, 0.2){$ \mathcal{A} $}
		\put(-0.02, 0.2){$ \mathcal{A}$}
			\put(-0.96, 0.33){$ \exp $}
		\put(-0.12, 0.33){$ \exp$}
				\put(-0.6, 0.62){$ z\mapsto e^{i\theta}z $}
		\put(-0.6, 0.17){$ z\mapsto z+\theta i$}
	\caption{\footnotesize Sketch of a Herman ring and its universal covering map, giving rise to its associated inner function, for some $\theta\in \mathbb{R}\smallsetminus\mathbb{Q}$.}\label{fig-HermanRing}
\end{figure}
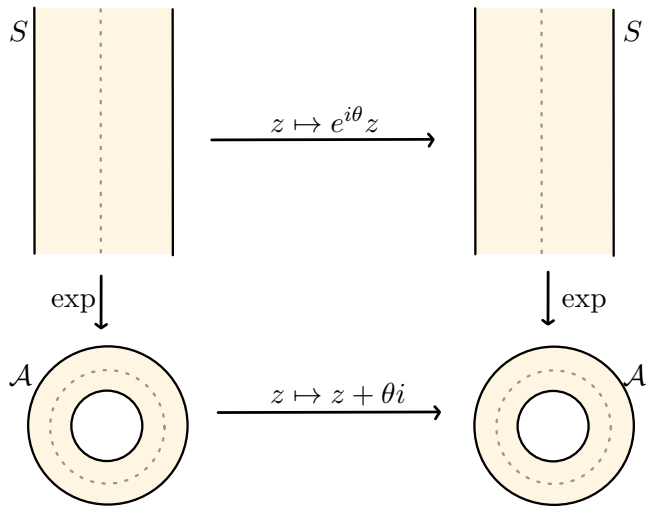

Concerning the boundary dynamics from a measure-theoretical point of view, Herman rings exhibit a surprising behaviour. Indeed, when restricted to one boundary component, the map is ergodic and recurrent, but not exact (since it is an irrational rotation). However, for the associated inner function, the map is non-ergodic and non-recurrent (even when restricted to the preimage of one boundary component).

We note that, if the radial extension of the associated inner function is exact (resp. ergodic, recurrent), the same holds for the  boundary map on the Fatou component. Hence, Herman rings show that the converse is not true in general.

We shall emphasize that, in the particular case of simply connected Fatou components, the ergodic properties of the radial extension of the associated  inner function transfer to the boundary map on the Fatou component \cite[Thm. A]{Jov24-boundaries}. Hence, the relationship between the ergodic properties of the radial extension of the associated  inner function and the boundary map on the Fatou component is less transparent in the multiply connected case, and worth of study.

\subsection{Doubly connected wandering domains of entire functions}\label{subsect-doubly-connWD} It is well-known that entire functions may have doubly connected wandering domains. More precisely, if $U$ is a doubly connected wandering domain of an entire function, then $ U_n=f^n(U) $ is doubly connected for all $n\geq 0$ (see e.g. \cite{KS08}). Then,  $ f|_{U_n} $ is proper and with no critical values. Indeed, $ f|_{U_n} $  is conjugate to a power map $ z\mapsto z^d $, and the associated inner function can be taken to be the identity
(see e.g. \cite[Sect. 3]{Ferreira_MCWD1} and compare with Fig. \ref{fig-DCWanderingDomain}).
Since $ f|_{U_n} $  is conjugate to $ z\mapsto z^d $, $ f $ is exact when restricted to each boundary component (even though the associated inner function is clearly non-exact).

\begin{figure}[h]\centering
	\includegraphics[width=8cm]{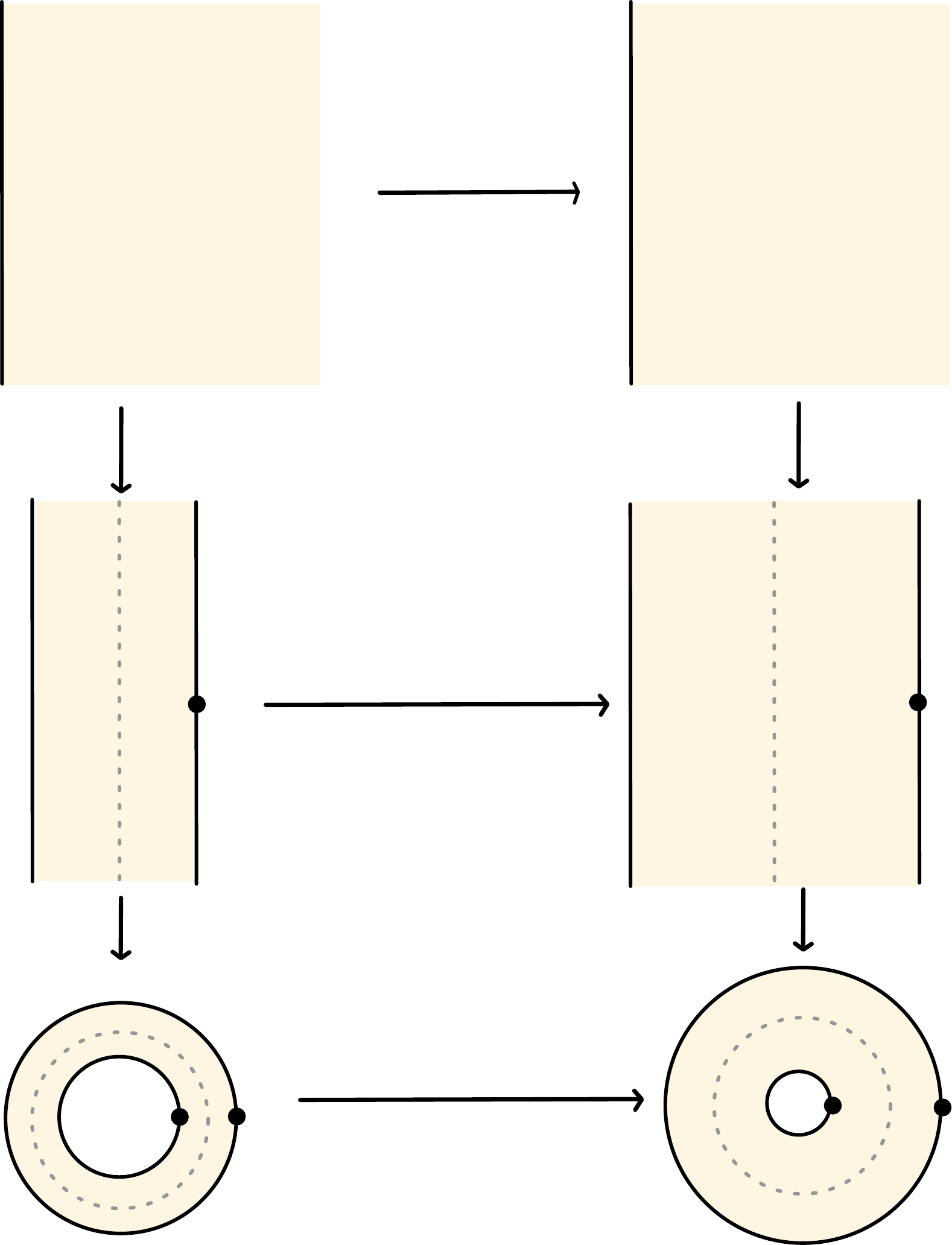}
	\setlength{\unitlength}{8cm}
		\put(-1.01, 0.75){$ S $}
	\put(-0.02, 0.75){$ S $}
			\put(-1.04, 1.25){$ \mathbb{H} $}
	\put(-0.04, 1.25){$ \mathbb{H}$}
	\put(-1.03, 0.2){$ \mathcal{A} $}
	\put(-0.02, 0.2){$ \mathcal{A}$}
	\put(-0.96, 0.33){$ \exp $}
	\put(-0.12, 0.33){$ \exp$}
		\put(-0.87, 0.56){\tiny $ \ln R $}
	\put(-0.02, 0.56){\tiny $ 2\ln R$}
			\put(-0.74, 0.12){\tiny $ R $}
	\put(-0.0, 0.12){\tiny $  R^2$}
		\put(-0.89, 0.12){\tiny $1/ R $}
	\put(-0.12, 0.12){\tiny $ 1/ R^2$}
		\put(-1.11, 0.83){$ z\mapsto e^{\frac{i\pi}{2\ln R}z} $}
	\put(-0.14, 0.83){$ z\mapsto e^{\frac{i\pi}{\ln R}z} $}
	\put(-0.6, 0.58){$ z\mapsto 2z $}
		\put(-0.53, 1.12){$  \textrm{id}$}
	\put(-0.6, 0.17){$ z\mapsto z^2$}
	\caption{\footnotesize Skecht of a doubly connected wandering domain of an entire function, and its associated inner function.}\label{fig-DCWanderingDomain}
\end{figure}
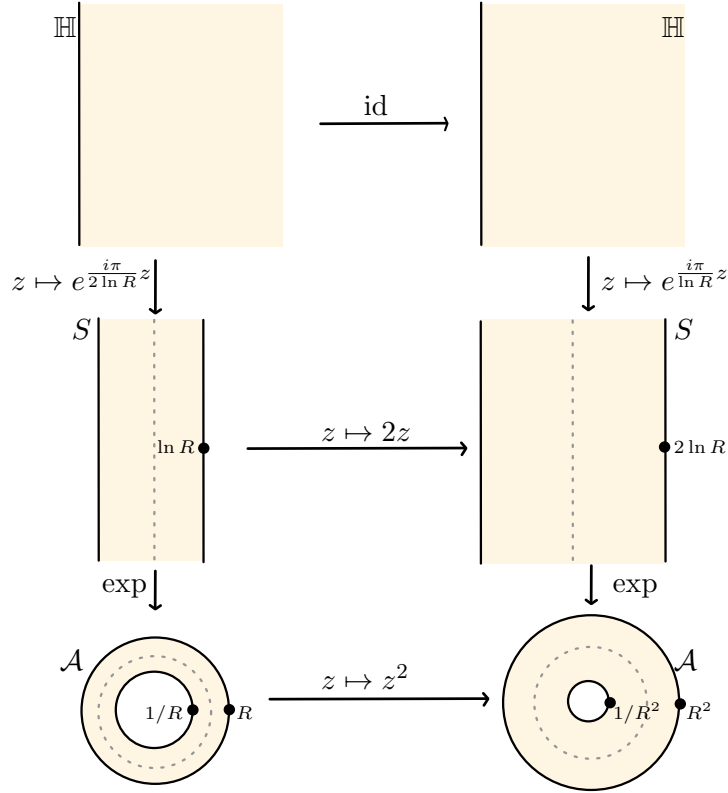

We shall emphasize that for both Herman rings and doubly connected wandering domains the associated inner function is a conformal automorphism of $\mathbb{D}$ (and hence $ f $ acts as a local hyperbolic isometry in their interior). However, the boundary map of a Herman ring is non-exact (even when restricted to a boundary component), while for doubly connected wandering domains it is exact when restricted to a boundary component.
\subsection{Baker's doubly connected basin of attraction}\label{subsect-baker} Consider the self-map of $ \mathbb{C}^* $ \[ f(z)= e^{\alpha z-\alpha/z}, \hspace{0.5cm} \alpha\in (0, 1/2),\] considered by Baker in \cite{Baker-puncturedplane}.
For this map, $ \mathcal{F}(f) $ consists of a unique Fatou component, which is the  basin of attraction of 1. The points 0 and $\infty$ are the only asympotic values of $ f $ (which are also the essential singularities of $ f $), $i$ and $-i$ are its critical values, contained in the basin.

For this specific example, we can compute the associated inner function explicitly. Indeed, $f(z)= e^{\alpha z-\alpha/z}$ is semiconjugate to $ F(z)= 2\alpha \sin z$ by $ z\mapsto e^{iz}$, that is
	\[\begin{tikzcd}[row sep=1cm, column sep = 2 cm]
	\mathbb{C} \arrow{r}{F(z)= 2\alpha \sin z} \arrow[swap]{d}{z\mapsto e^{iz}}&  	\mathbb{C} \arrow{d}{z\mapsto e^{iz}} \\	
		\mathbb{C}^* \arrow{r}{f(z)= e^{\alpha z-\alpha/z}}& 	\mathbb{C}^*
\end{tikzcd}
\] 
The Fatou set of $F$ is connected and consists of the immediate basin of attraction of 0, say $V$. The associated inner function for $F|_V$ has been computed explicitly in \cite[Thm. 1.8]{EvdoridouVasiliki2020FA}. Indeed, let $ \varphi\colon \mathbb{D}\to V $ be the Riemann map with $ \varphi(0)=0 $ and $ \varphi'(0)>0 $, then the associated inner function is the infinite Blaschke product $ B\colon\mathbb{D}\to\mathbb{D} $, \[ B(z)=z\prod_{n\geq 1}\frac{a_n^2-z^2}{1-a^2_n z^2}, \hspace{0.8cm} a_n=\frac{\tau^n-1}{\tau^n+1},\]where $\tau$ is uniquely determined by $\alpha$.
Note that $ \pi\colon\mathbb{D}\to U $ defined as $ \pi(z)=e^{i\varphi(z)} $ is a universal covering map of $U$. Hence, $ B $ is also the inner function associated with $f|_U$, as the following commutative diagram shows.

	\[\begin{tikzcd}[row sep=1cm, column sep = 2 cm]
	\mathbb{D} \arrow{r}{B} \arrow[swap]{d}{ \varphi}\arrow[dd, shift right=5, bend right=30, swap, "\pi"]&  	\mathbb{D} \arrow{d}{\varphi}\arrow[dd, shift left=5, bend left=30, "\pi"] \\	
	V \arrow{r}{F(z)= 2\alpha \sin z}\arrow[swap]{d}{ e^{iz}}& V \arrow{d}{e^{iz}} 
	\\	
	U \arrow{r}{f(z)= e^{\alpha z-\alpha/z}}& U
\end{tikzcd}
\] 

Note that 1 and $-1$ are the singularities of the inner function, which coincide with the limit set of $ \pi $. Compare with Figure \ref{fig-Baker}.

\begin{figure}[!h]\centering	\captionsetup[subfigure]{labelformat=empty,justification=centering}
	\begin{subfigure}{0.45\textwidth}
	\includegraphics[width=\textwidth]{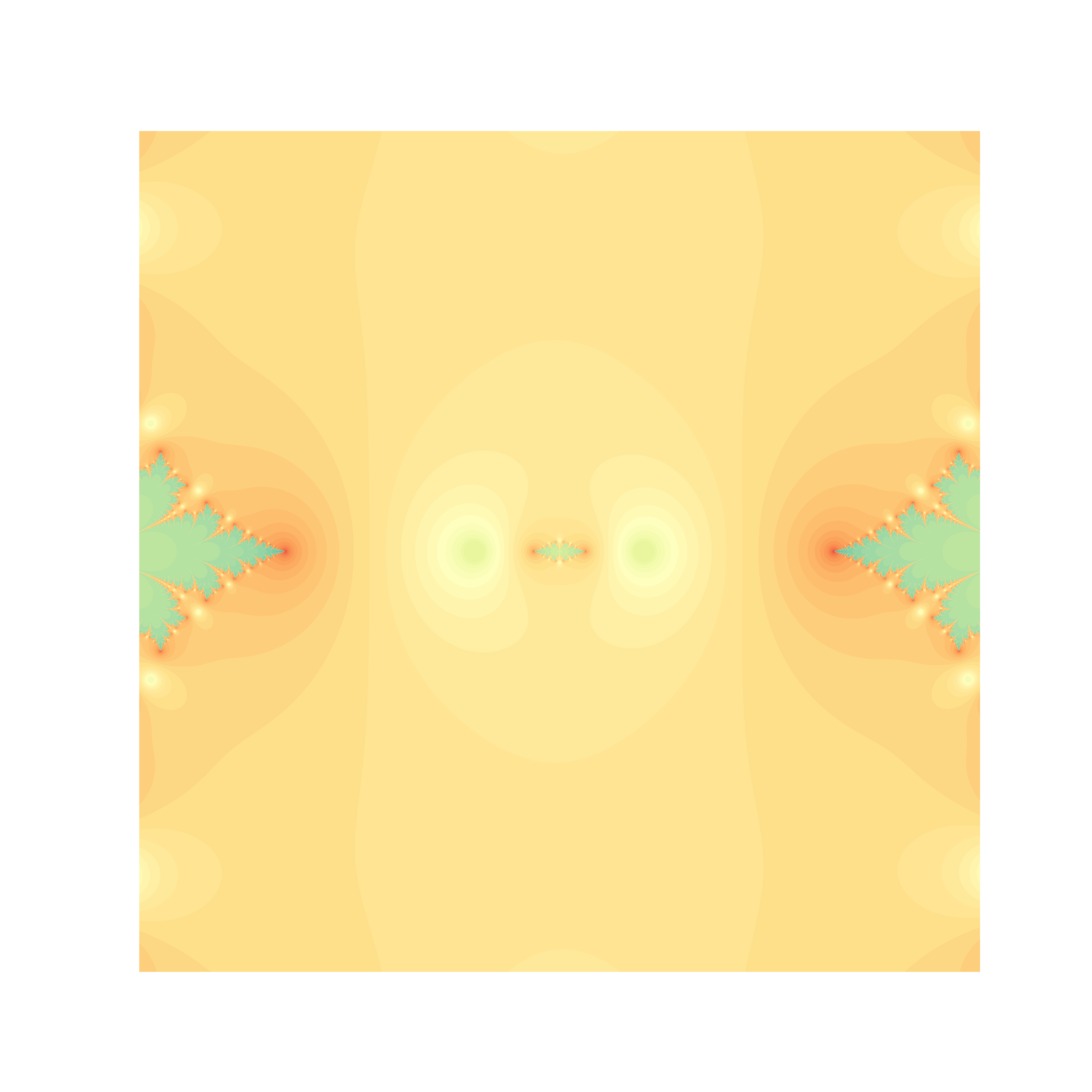}
		\caption{Dynamical plane of Baker's example with $\alpha = 0.4$.}
	\end{subfigure}
	\begin{subfigure}{0.45\textwidth}
	\includegraphics[width=\textwidth]{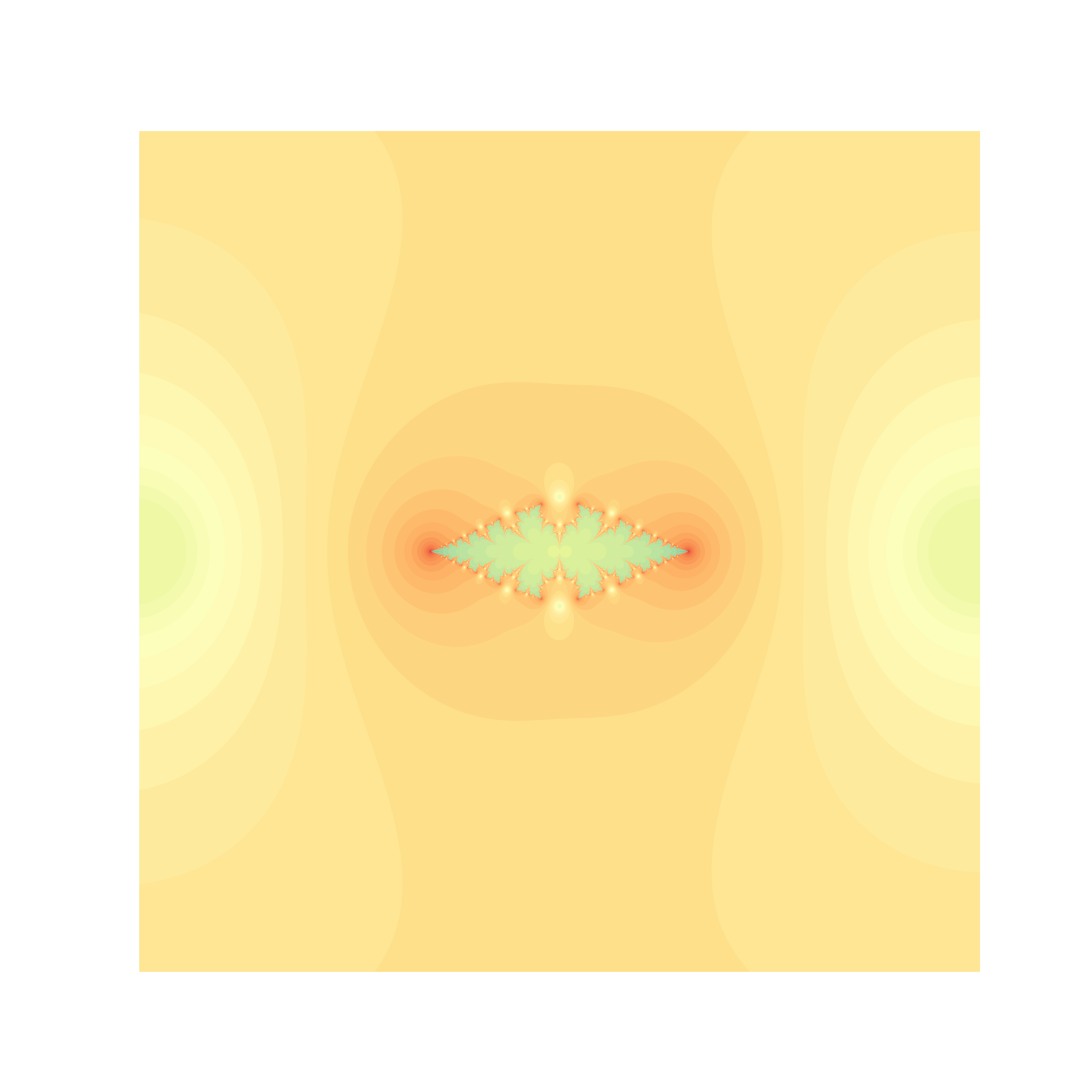}
		\caption{\footnotesize Enlarged version of the Cantor bouquet around zero in the previous figure.}
	\end{subfigure}

	\begin{subfigure}{0.85\textwidth}
		\includegraphics[width=\textwidth]{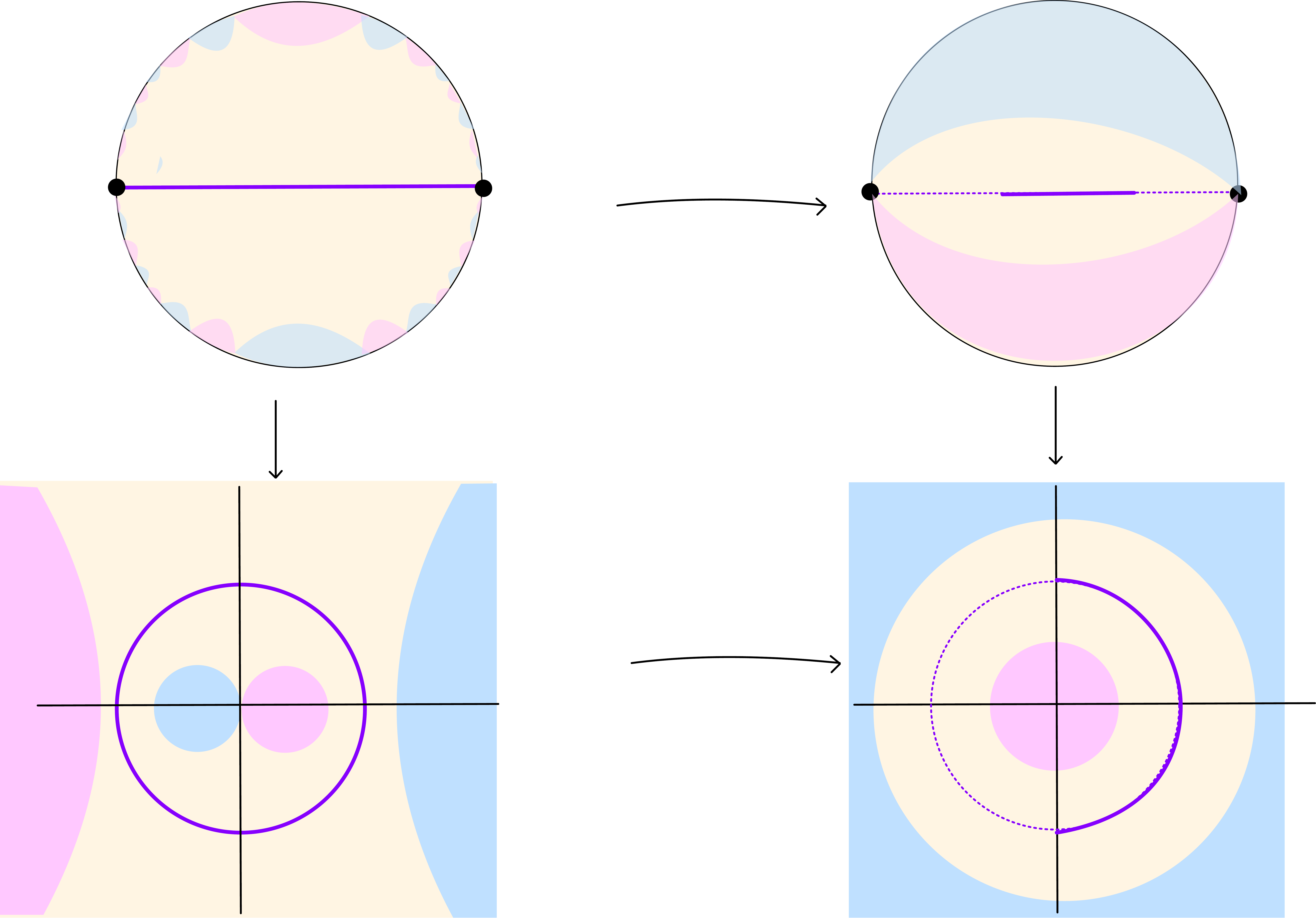}
			\setlength{\unitlength}{\textwidth}
		\put(-0.9, 0.65){$ \DD$}
		\put(-0.09, 0.65){$\DD  $}
		\put(-1.03, 0.3){$ U $}
		\put(-0.02, 0.3){$ U$}
		\put(-0.82, 0.36){$ \pi$}
		\put(-0.19, 0.36){$ \pi$}
		\put(-0.48, 0.55){$ B $}
		\put(-0.48, 0.2){$f$}
		\caption{\footnotesize  Sketch of the dynamics in Baker's doubly connected basin of attraction, together with the mapping properties of the associated inner function.}
	\end{subfigure}
	\caption{\footnotesize }\label{fig-Baker}
\end{figure}


Note that, in contrast with rational maps, the limit set is not $\partial\mathbb{D}$. Indeed, due to the presence of essential singularities on the boundary of $U$, ergodicity and recurrence are  compatible with having regular points for $ \pi $. More precisely, the proof of Theorem \ref{thmB} showing that $ \Lambda= \partial\mathbb{D}$ relied on the fact that the image of a connected component of the boundary is again connected. Now, due to the presence of essential singularities on $ \widehat{\partial} U $, this is no longer true: the image of a true crosscut neighbourhood eventually covers the whole boundary (compare again with Figure \ref{fig-Baker}).

Finally, we should remark that the example of a doubly connected  basin of attraction of $$ f_{\alpha, \lambda}(z)=ze^{\alpha(z+1/z)+\lambda}, $$ first studied by Keen \cite{Keen}, presents similar dynamics and the same pathlogical features of the associated inner function. We shall not develop this example in detail here, since the ideas are the same as above and the computations more involved.

\section{Multiply connected basins in class $\mathbb{K}$. Theorem \ref{thmC}}\label{section-basins-trans}

In this Section we shall prove the following version of Theorem \ref{thmC}, which is more complete than the one stated in the introduction.
\begin{setting}\label{settingInvK} Let $ f\in\mathbb{K} $, and let $ U$ be an attracting or parabolic basin. Assume $ U $ is multiply connected.
	Consider  $\pi\coloneqq \pi_U\colon\DD\to U $ a universal covering map, and let $ g\colon\mathbb{D}\to\mathbb{D} $ be such that $ \pi \circ g= f\circ \pi$.
\end{setting}
\begin{thm}{\bf (Theorem C- general version)}\label{thmC-general-version}
	In the Setting \ref{settingInvK}, the following hold true.
		\begin{enumerate}[label={\normalfont(\alph*)}]
		\item\label{thmC-a} The associated inner function $ g $ has infinite degree, and $ \Lambda\subset E(g) $. Moreover, if $ U $ is of infinite connectivity, then $ E(g) $ is uncountable.
		\item\label{thmC-b} If $ \xi\in\partial\mathbb{D}$ is of bounded type, there exists $ n\geq0 $ such that $ (g^*)^n(\xi) $ is not well-defined. 
			\item\label{thmC-c} $ f|_{\partial U} $ is exact (and, in particular, ergodic) and recurrent. Moreover, for $ \omega_U $-almost every $ x\in\partial U $, $ \left\lbrace f^n(x)\right\rbrace _n $ is dense on $ \partial U $.
	\end{enumerate}
Moreover, assume $ SV(f|_U) $ is compactly contained in $U$, then
\begin{enumerate}[label={\normalfont(\alph*)}]
  \setcounter{enumi}{3}
	\item\label{thmC-d} either $ \lambda (E(g))=  \lambda (\Lambda)=0 $, or $ E(g)=\Lambda=\partial\mathbb{D} $.
\end{enumerate}

\end{thm}
\begin{proof} We shall prove the statements separately.
	
	\begin{enumerate}[label={\normalfont(\alph*)}]

		\item The first statement is a reformulation of Lemma \ref{lemma-sing-limit-sets}{\em \ref{lemma-singularities-b}}, taking into account that, for the considered Fatou components, $ SV(f|_U)\neq \emptyset $. For the second statement, note that $ \Lambda\subset E(g) $ (Lemma \ref{lemma-sing-limit-sets}{\em \ref{lemma-singularities-c}}), and  if $ U $ is of infinite connectivity, $ \Lambda $ is uncountable.
		\item The proof is the same as in  Theorem \ref{thmB}{\em\ref{thmB-itemc}} (note that in this specific item we did not use the properness of $f|_U$).

			\item Since $U$ is an attracting or parabolic basin, then $ g^*|_{\partial\mathbb{D}} $ is exact (and, hence, ergodic), and recurrent (by Thm. \ref{thm-ergodic-properties} and Prop. \ref{prop-inner-rational}). Thus, $ f|_{\partial U} $ is exact, ergodic and recurrent. Ergodicity and recurrence already imply that $\omega_U$-almost every orbit is dense on $\textrm{supp } \omega_U$ (this is a standard fact; see e.g. \cite[Prop. 1.2.2]{Aaronson97}). Since $  \textrm{supp } \omega_U=\partial U$ (Thm. \ref{thmA}), we get the desired result.
	\end{enumerate}

Now we work under the assumption that $SV(f|_U)$ is compactly contained in $U$.
\begin{enumerate} [label={\normalfont(\alph*)}]
	\setcounter{enumi}{3}
\item Let us first show that, either $ \Lambda $  has zero Lebesgue measure, or $\Lambda=\partial \mathbb{D}$. 

Assume temporarily that $U$ is a basin of attraction, and choose $g$ so that $ g(0)=0 $. Then, the Lebesgue measure $\lambda$ is invariant under $g^*$. Assume  $\Lambda$ is a Cantor set, and let us prove that it has zero measure. By Proposition \ref{prop-invariance-limit-sets}, there exists a subset $ Z\subset \partial\mathbb{D}\smallsetminus\Lambda $ such that $$ \lambda (Z)=\lambda (\partial\mathbb{D}\smallsetminus\Lambda)=\lambda((g^* )^{-1}(Z)),$$ and $ (g^* )^{-1}(Z)\subset \partial\mathbb{D}\smallsetminus\Lambda$. Thus, $$ Z_\infty\coloneqq \bigcap\limits_{n\geq 0} (g^*)^{-n}(Z)$$ satisfies that $ (g^* )^{-1}(Z_\infty)=Z_\infty$ and $ \lambda(Z_\infty)=  \lambda (\partial\mathbb{D}\smallsetminus\Lambda)$. Since $ g^*|_{\partial \mathbb{D}} $ is ergodic and $ \lambda(Z_\infty)>0 $, this implies that $ \lambda(Z_\infty)=1 $. Then, $ \Lambda $ is a Cantor set of zero measure, as desired. 

	In the case when $U$ is a parabolic basin, then $\lambda$ is no longer invariant (since the Denjoy-Wolff point is in $\partial\mathbb{D}$). However, $ g^* $ admits an invariant $\sigma$-finite measure $\mu$ which is absolutely continuous with respect to $ \lambda $ \cite[Thm. C]{DM91}. For such a measure, one can repeat the previous procedure and prove that, if $\Lambda$ is a Cantor set, then it has zero measure. Hence, we proved that either $ \lambda(\Lambda) =0$, or $\Lambda=\partial \mathbb{D}$.

		Combining the fact that  $ \lambda(\Lambda)=0 $ or $ \Lambda=\partial\mathbb{D} $ with the inclusion $ \Lambda\subset E(g) $ (Lemma \ref{lemma-sing-limit-sets}{\em\ref{lemma-singularities-c}}), we are left to prove that, if $ \lambda(\Lambda)=0 $, then $ \lambda (E(g))=0 $.
	
	First note that, since $  SV(f|_{U})$ is compactly contained in $ U $ by assumption, it follows that $ SV(g)\cap \partial \mathbb{D}\subset \Lambda $ (and, hence, $ \lambda(SV(g)\cap \partial \mathbb{D})=0 $). Indeed, if $ \xi\notin \Lambda $, then there exists a crosscut neighbourhood $ N $ around $ \xi $ whose image under $ \pi $ is a simply connected crosscut neighbourhood in $ U $. Since $  SV(f|_{U})$ is compactly contained in $ U $, we can assume that $ \pi(N) $ does not contain singular values. Hence, $ N $ does not contain singular values for $ g $, and it is a regular value for $ g $ (by Prop. \ref{prop-SVinner}). 
	
	Therefore, $ \partial\mathbb{D}\smallsetminus\Lambda $ has full $ \lambda $-measure, and all inverse branches are well-defined in $ \partial\mathbb{D}\smallsetminus\Lambda $ and conformal. Thus, $ g^{-1}(\partial\mathbb{D}\smallsetminus\Lambda) $ has full $ \lambda $-measure (because radial extensions of inner functions are non-singular, Thm. \ref{thm-ergodic-properties}) and consists of points which are not singularities (since points in $ g^{-1}(\partial\mathbb{D}\smallsetminus\Lambda) $ are mapped locally conformally to points in $ \partial\mathbb{D}\smallsetminus\Lambda $). This ends the proof of the statement.
\end{enumerate}
\end{proof}

\begin{obs}{\em (Baker domains)}\label{remark-dp-Baker-domains}
In this paper we do not consider Baker domains (i.e. invariant Fatou components in which iterates converge to an essential singularity), since they exhibit much wilder dynamics, even in the simply connected case. However, for a specific type of Baker domains, the so-called {\em doubly parabolic Baker domains}, most of the properties proved in Theorem \ref{thmC-general-version} still hold. Indeed, doubly parabolic Baker domains are characterized as those for which the associated inner function $ g $ is doubly parabolic (c.f. \cite{BFJK_AbsorbingSets}). Thus, $g^*|_{\partial\mathbb{D}}$ is ergodic, so \ref{thmC-a} and \ref{thmC-d} hold. In \cite{BFJK_AbsorbingSets} it is proved that such Baker domains admit always a simply connected absorbing domain, so \ref{thmC-b} also holds. Concerning \ref{thmC-c}, $g^*|_{\partial\mathbb{D}}$  may not be recurrent (in \cite{BFJK-Escaping} examples are given in the simply connected case).
\end{obs}

\section{Multiply connected wandering domains. Theorem \ref{thmD}}\label{section-WD}
We start by recalling some notation, as well as the statement of Theorem \ref{thmD}. In this section, we assume that $f$ is in class $\mathbb{K}$, and $U_0$ is a multiply connected wandering domain of $f$ (notice that its iterates $U_n$ might not be multiply connected). We take universal coverings $\pi_n\colon\DD\to U_n$, $n\geq 0$, and lifts $g_n$ of $f|_{U_n}$, obtaining the following commutative diagram.
\[ \begin{tikzcd}
\DD \arrow{r}{g_0} \arrow{d}{\pi_0}&  \DD \arrow{r}{g_1} \arrow{d}{\pi_1}&  \DD \arrow{r}{g_2} \arrow{d}{\pi_2}&  \DD \arrow{r}{g_3} \arrow{d}{\pi_3}&\dots \\
U_0 \arrow{r}{f}& U_1\arrow{r}{f}& U_2\arrow{r}{f}& U_3\arrow{r}{f}& \dots
\end{tikzcd} \]
Furthermore, the inner functions $g_n\colon\mathbb{D}\to\mathbb{D}$ can be normalised to fix the origin. With this setup, Theorem \ref{thmD} states that
\begin{enumerate}[label={(\alph*)}]
	\item if $U_0$ is infinitely connected, its limit set $ \Lambda_0 $  is either a Cantor set of zero measure, a Cantor set of positive measure, or $ \partial\mathbb{D} $. 
	
	\noindent For each of these possibilities, there exists a transcendental meromorphic function with a multiply connected wandering domain  with a limit set of the corresponding type.
\end{enumerate}
Under the additional assumption that $U_n$ is bounded for all $ n\geq 0 $, then
		\begin{enumerate}[label={(\alph*)}]   \setcounter{enumi}{1}
		\item the limit set $ \Lambda_n$ of $ U_n $ satisfies that $ \lambda (\Lambda_n)= \lambda (\Lambda_0)$, for all $ n\geq 0 $.
		\item If $ \sum_n 1-\left| g'_n(0)\right| =\infty $, then $f|_{\partial U_n}$ is exact and $ \Lambda_n=\partial\mathbb{D} $, for all $ n\geq 0 $.
	\end{enumerate}
Finally, Theorem \ref{thmD} states that there exists a transcendental meromorphic function with a wandering domain whose exact components are singletons. With all this in mind, let us start the proof.

\begin{proof}[Proof of Theorem \ref{thmD}]
The first claim of (a) is simply a restating of a well-known fact about limit sets of deck transformation groups; see e.g. \cite[p. 76]{Hub06}. The remarkable fact is that all cases can be realized, although examples of wandering domains with Cantor set limit sets with zero measure are already known -- see e.g. the finitely connected examples in \cite{BKL90}. 

For the cases with positive measure, we need only apply the approximation results of \cite{MartiPeteRempeWatermann-mero}, which say that every regular plane domain can be realized as a wandering domain of some transcendental meromorphic function. Therefore, the problem reduces to finding regular domains to which approximation theory can be applied. In the case of domains whose limit set is a Cantor set of positive measure, we have {\em champagne subregions}: subdomains of $\DD$ that can be rigged to have positive harmonic measure at $\partial\DD$ (see \cite{OCS04}), giving us fat Cantor sets as limit sets. More precisely, a champagne subregion can be defined as
\[ \Omega := \DD\setminus \bigcup_{n\geq 1} D(w_n, r_n); \]
necessary and sufficient conditions on the points $w_n$ and the radii $r_n$ are given in \cite[Thm. 2.1]{OCS04} for the boundary component $\partial\DD\subset\partial\Omega$ to have positive harmonic measure in $\Omega$. Furthermore, the domain $\Omega$ is clearly regular, i.e. satisfies $\mathrm{int}(\overline{\Omega}) = \Omega$. For limit sets equal to $\partial\DD$, we take as a regular domain an infinitely connected  basin of attraction of a rational map (recall Theorem \ref{thmB}). This finishes the proof of (a).

Item (b) is an easy consequence of Corollary \ref{cor-invariance-limit-sets-proper}.

For item (c), we use the results of \cite{Pom81} (see also \cite[Lemma 7.5]{BEFRS2}). More specifically, the hypothesis that $\sum_{n} (1 - |g_n'(0)|) = \infty$ implies that the boundary dynamics of the sequence $g_n^*$, or equivalently of the compositions $G_n^* = g_n^*\circ\cdots g_0^*$, is exact. That is, for every measurable set $ X\subset \partial \DD $ such that for all $ n$, $X=(G^*_n)^{-1}(X_n)$ for some measurable $X_n\subset \DD$, then $\lambda(X)=0$ or $\lambda(X)=1$.  Since $\Lambda_0 = (G_n^*)^{-1}(\Lambda_n)$ for all $n$ (up to a set of measure zero), we must have either $\lambda(\Lambda_0) = 0$ or $\lambda(\Lambda_0) = 1$. Exactness of $f|_{\partial U}$ also follows from the exactness of the sequence $ \left\lbrace G_n^*\right\rbrace _n $. The fact that $\Lambda_n = \partial \DD$ can now be proved exactly as for Theorem \ref{thmB}{\em \ref{thmB-itema}}.

To show that there exists a transcendental meromorphic function whose exact components are singletons just take any regular bounded domain $ U $ and realize it as a wandering domain of a transcendental meromorphic function, using \cite{MartiPeteRempeWatermann-mero}. It follows from their construction that $ f^n(\overline{U}) $ is univalent for all $n\geq 0$, and this already implies that the exact components are singletons. 
\end{proof}

Lastly, it is well-known that multiply connected wandering domains of entire functions exhibit more restrictive dynamics. In this particular case, we prove the following. Regarding item \ref{thmEc}, we point out that it may fail for infinitely connected wandering domains in a very peculiar way: the outer boundary component of $U_n$ may coincide with the inner boundary component of $U_{n+1}$, ``merging'' the two exact components into one. This phenomenon was first observed by Bishop \cite{Bis18} when constructing transcendental entire functions with Julia sets of Hausdorff dimension one, and was later shown by Baumgartner \cite{Bau15} to occur for other transcendental entire functions as well.
\begin{thm}{\bf (Multiply connected wandering domains of entire functions)}\label{thmE}
	Let $ f $ be an entire function, and let $ U $ be a multiply connected wandering domain. Let $ \pi_U\colon\mathbb{D}\to U $, $ \pi_{f(U)}\colon\mathbb{D}\to f(U) $  be a universal covering maps, and let $ g\colon\mathbb{D}\to\mathbb{D} $ such that $ \pi_{f(U)}\circ g=f\circ \pi_U $. Then,
		\begin{enumerate}[label={\normalfont (\alph*)}]
		\item\label{thmEb} $ g^*(\xi) $ exists for all $ \xi\in\partial \mathbb{D} $;
		\item\label{thmEc} if $U$ is doubly connected, then $ f|_{\partial U} $ has two exact components.
	\end{enumerate}
\end{thm}
\begin{proof}
Multiply connected domains of entire functions are always bounded \cite{Baker84}. Therefore, $ f\colon U\to f(U) $ is a proper map, and all the specific machinery developed in Section \ref{subsect-mapping-properties} can be applied.
Moreover, if we denote by $\left\lbrace U_n\right\rbrace _n$ the orbit of the wandering domain $U$, i.e. $ U_{n+1}=f(U_n) $, then either $U_n$ is doubly connected for all $n\geq n_0$, or $U_n$ is infinitely connected for all $n\geq0$. In the first case, $ f|_{U_n} $ is conjugate to a power map between annuli. We refer to \cite{BergweilerRipponStallard_MCWD} for details.
	

For item {\em \ref{thmEb}}, we recall Lemma \ref{lemma-mapping-properties}; we only have to show that $g^*(\xi)$ is defined even for points of bounded type.  Let $\xi\in\dD$ be of bounded type and let $ R_\xi= \left\lbrace t\xi \colon t\in \left[ 0,1\right)  \right\rbrace $. 
 Since $ \xi$ is of bounded type, $ \xi  $ belongs to the non-tangential limit set of some finitely generated subgroup of the group of  deck transformations $ \Gamma $ \cite[Summary 4.15]{FerreiraJove}, whose generators correspond to the closed loops $ \sigma_1, \dots, \sigma_m $.  Then $\gamma = \pi_U(R_\xi)$ is a word on the finitely many generators $\sigma_1, \ldots, \sigma_m$ of the fundamental group $\pi(U)$.

Since $f$ is entire, the image of a non-contractible curve in $U$ is again non-contractible in $f(U)$ (this follows easily from the Maximum Modulus Principle and the complete invariance of the Fatou and Julia sets). Thus, $f(\gamma)$ is a word on the finitely many generators $f(\sigma_1), \ldots, f(\sigma_m)$ of the fundamental group $\pi(f(U))$. Lifting $f(\gamma)$ yields the curve $g(R_\xi)$, which lands non-tangentially at a point in the non-tangential limit set of the finitely generated subgroup of the deck transformation group corresponding to $f(\sigma_1), \ldots, f(\sigma_m)$. This point, by definition, is $g^*(\xi)$.

Lastly, for item {\em \ref{thmEc}}, if $U$ is doubly connected, then so is each $U_n$, and $f|_{U_n}$ reduces to a power map (see e.g. \cite{KS08,BergweilerRipponStallard_MCWD,Ferreira_MCWD1}); exactness of each boundary component follows easily.
\end{proof}


\end{document}